\documentclass[12pt]{article}
\usepackage{amsmath,amssymb,amsfonts,amsthm,color,mathtools}
\usepackage{graphics}
\usepackage{epsfig}
\usepackage{mathrsfs}
\usepackage[bookmarks,bookmarksopen=true,bookmarksnumbered=true]{hyperref}
\usepackage{geometry}
\newtheorem{thm}{{Theorem}}[section]
\newtheorem{prop}{{Proposition}}[section]
\newtheorem{lem}{{Lemma}}[section]
\newtheorem{rem}{{Remark}}[section]
\newtheorem{example} {Example}[section]

\renewcommand\qed{$\blacksquare$}

\def\R{{\mathbb{R}}}

\def\CC{{\rm \kern.24em \vrule width.02em height1.4ex depth-.05ex \kern-.26emC}}

\catcode`\@=11

\newcommand{\eps}{\varepsilon}

\def\theequation{\@arabic{\c@section}.\@arabic{\c@equation}}

\begin{document}

\title{Homogenization for non-local elliptic
operators in both perforated and non-perforated domains}

\author{Loredana B\u alilescu\thanks{\noindent Department of Mathematics and Computer Science, University of Pite\c{s}ti, 110040 Pite\c{s}ti,
Str. T\^argu din Vale nr.1, Arge\c{s}, Romania and Department of Mathematics, Federal University of Santa Catarina, Brasil; {\tt smaranda@dim.uchile.cl} }
\qquad 
Amrita Ghosh\thanks{\noindent  Department of Mathematics, Universit\'e de Pau et des Pays de l'Adour,  France {\tt amrita.ghosh@univ-pau.fr }}
\qquad 
\quad Tuhin Ghosh\thanks{\noindent Department of Mathematics, University of Washington; {\tt tuhing@uw.edu}}
}

\date{}

\maketitle
\vspace*{-2ex}
\abstract{
\noindent
In this paper, we focus on the homogenization process of the  non-local elliptic boundary value problem
$$\mathcal{L}_\eps^s u_\eps =(-\nabla\cdot (A_\eps(x)\nabla))^{s}u_\eps=f \mbox{ in } \mathcal O,
$$
with $0<s<1$, considering non-homogeneous Dirichlet type condition outside of the bounded domain $\mathcal O\subseteq \R^n$. We find the homogenized problem by using the $H$-convergence method, as $\eps\to 0$, under  standard uniform ellipticity, boundedness and symmetry assumptions on coefficients $A_\eps(x)$, with the homogenized coefficients as the standard $H$-limit (cf. \cite{MT1})
of the sequence $\{A_\eps\}_{\eps>0}$. We also prove that the commonly referred to as \textit{the strange term} in the literature   (see \cite[Chapter 4]{MT}) does not appear  in the homogenized problem associated with  the fractional Laplace operator $(-\Delta)^s$ in a perforated domain. Both of these results have been obtained in the class of general microstructures. Consequently, we could certify that the homogenization process, as $\varepsilon\to 0$, is stable under $s\to 1^{-}$ in the non-perforated domains, but not necessarily in the case of perforated domains.     
}

\vskip .5cm
\noindent

\section{Introduction}

The general question tackled in this paper is the
homogenization process of Dirichlet type problem associated with fractional elliptic non-local
operator in bounded 
domains. Precisely, let $\mathcal{L} =-\nabla\cdot (A(x)\nabla)$ be the classical uniformly elliptic
operator in divergence form  with the anisotropic matrix valued function $A(x)$ defined in whole space $\mathbb{R}^n$. Then, for
$0<s<1$, we consider the fractional non-local
operator (for the definition, see Section \ref{Sect2} below):
$$\mathcal{L}^s =(-\nabla\cdot (A(x)\nabla))^{s}.$$
We are interested in the  restriction of the fractional  Laplacian $\mathcal{L}^s$ in a bounded domain $\mathcal{O}\subset\mathbb{R}^n$ 
and the associated non-homogeneous Dirichlet exterior boundary value problem $\mathcal{L}^su =f$ in a  smooth enough  bounded domain $\mathcal{O}\subset\mathbb{R}^n$ with $u=g$ in $\mathbb{R}^n\setminus \mathcal{O}$. 

These kind of fractional and non-local operators often arise in problems modelling diffusion process, ergodic random environments and random processes with jumps, enabling possible applications in probability theory, physics, finance, and biology, to name a few (for more details, see the survey works \cite{BV,RX}). In particular, the above operator $\mathcal{L}^s$ as a  linear integro-differential operator (see \eqref{t45}) could be considered as an
infinitesimal generator of generalized L\'{e}vy processes of the probabilistic/stochastic model under consideration with a random process that allows long jumps with a polynomial tail (see the books \cite{APP,BJ, EK}). 
For example, if $g=0$, probabilistically it represents  the  infinitesimal  generator of a symmetric $2s$-stable L\'{e}vy process that particles are killed upon leaving the domain $\mathcal{O}$. 


The paper aims at providing a macro scale approximation
to a problem with heterogeneities/microstructures at micro scale $\eps$ by suitably averaging out
small scales $(\eps\to 0$) and by incorporating their effects on large scales. These effects are quantified by the so-called homogenized coefficients \cite{A,BLP,JKO,T}.
We will be using the $H$-convergence method (for more details on $H$-limits, we refer to \cite{A,MT-Hconv,T}), under  standard uniform ellipticity, boundedness and symmetric assumptions
on the coefficient matrices $\{A_\eps(x)\}_{\eps>0}$. 

More precisely, let us consider $s\in(0,1)$, $\mathcal{O}\subset\mathbb{R}^n$ be a  bounded Lipschitz domain. For each $\eps>0$, consider $u_\eps\in H^s(\mathbb{R}^n)$, which is the solution to the following non-local Dirichlet type problem:
\begin{equation}\label{t58}
\begin{cases}
\mathcal{L}_\eps^{s}u_\eps= \left(-\nabla\cdot(A_\eps(x)\nabla)\right)^s u_\eps=f & \mbox{ in \ensuremath{\mathcal{O}}},\\
u_\eps=g & \mbox{ in }\mathbb{R}^n\setminus\mathcal{O},
\end{cases}
\end{equation}
for some $f\in \widetilde{H}^{s}(\mathcal{O})^{*}$ (see \eqref{def:Hstilde} below for the definition of this space) and $g\in H^s(\mathbb{R}^n)$. 

Our main goal is to pass to the limit in the above problem \eqref{t58}, as $\eps\to 0$, and to find the limit equation or the homogenized problem. Our main finding is that the homogenized equation is governed by the non-local elliptic operator $$\mathcal{L}^s_{*}= \big(-\nabla\cdot A_{*}(x)\nabla\big)^s,$$ where $A_{*}(x)$ is the standard $H$-limit of the sequence $\{A_\eps(x)\}_{\eps>0}$ in $\mathbb{R}^n$  under the standard uniform ellipticity and boundedness hypotheses on $\{A_\eps(x)\}_{\eps>0}=\{(a^{ij}_\eps(x))_{1\leq i,j\leq n}\}_{\eps>0}$ given as
\begin{equation}\label{t51}
\begin{cases}
a^{ij}_\eps(x)=a^{ji}_\eps(x)\mbox{ for all }x\in\mathbb{R}^{n}, 1\leq i,j\leq n,\mbox{ and }
\\
\lambda^{-1}|\xi|^{2}\leq\sum\limits_{i,j=1}^{n}a^{ij}_\eps(x)\xi_{i}\xi_{j}\leq\lambda|\xi|^{2}\mbox{ for all }x\in\mathbb{R}^{n}, \eps>0, \mbox{ and for some }\lambda>0.
\end{cases}\end{equation}
Let us state our first main result concerning the  homogenization process for   fractional non-local elliptic
operators in non-perforated domain.
\begin{thm}\label{t9}
Let $s\in(0,1)$, $\mathcal{O}\subset\mathbb{R}^n$ be a bounded domain with sufficiently smooth boundary. We assume
that the sequence $\{A_\eps(x)\}_{\eps>0}$ satisfies condition \eqref{t51}.
For each $\eps>0$, let $u_\eps\in H^s(\mathbb{R}^n)$ be the solution of problem \eqref{t58}, for some fixed  $f\in \widetilde{H}^{s}(\mathcal{O})^{*}$ and $g\in H^s(\mathbb{R}^n)$.
 Then, as $\eps\to 0$, up to a subsequence, we have 
\[u_\eps \rightharpoonup u \mbox{ weakly in }H^s(\mathbb{R}^n),\]
with the limit $u\in H^s(\mathbb{R}^n)$  characterized as the unique solution of the following homogenized problem: 
\[
\begin{cases}
\mathcal{L}_{*}^{s}u= \left(-\nabla\cdot(A_{*}(x)\nabla)\right)^s u=f & \mbox{ in \ensuremath{\mathcal{O}}},\\
u=g & \mbox{ in }\mathbb{R}^n\setminus\mathcal{O},
\end{cases}
\]
 where $A_{*}(x)$ is the $H$-limit of the sequence $\{A_\eps(x)\}_{\eps>0}$ in $\mathbb{R}^n$, that is,  
\[A_\eps\nabla w_\eps \rightharpoonup A_{*}\nabla w\quad \mbox{ weakly in } L^2(\mathbb{R}^n)^{n},\] for all test sequences $w_\eps\in H^1(\mathbb{R}^n)$ satisfying 
\begin{align*}
w_{\eps} &\rightharpoonup w \quad\mbox{weakly in }H^1(\mathbb{R}^n),\\
-\nabla\cdot(A_\eps\nabla w_\eps)& \quad\mbox{ strongly convergent in } H^{-1}(\mathbb{R}^n).
\end{align*}

Moreover, we have the following flux and energy convergences, respectively, as $\eps\to 0$:
\begin{eqnarray}\label{t21}
\mathcal{L}^{s/2}_\eps u_\eps \rightharpoonup \mathcal{L}^{s/2}_{*}u \quad\mbox{ weakly in }L^2(\mathbb{R}^n),
\\
\|\mathcal{L}^{s/2}_\eps u_\eps\|_{L^2(\mathbb{R}^n)} \to 
\|\mathcal{L}^{s/2}_{*}u\|_{L^2(\mathbb{R}^n)}.
\label{t2100}
\end{eqnarray}
\end{thm}
\begin{rem}
The above result is also applicable for spectral non-local operator $\mathcal{L}^{s}_S$, which is defined by the normalized eigenfunctions and eigenvalues for the operator $\mathcal{L}$ in $\mathcal{O}$ with homogeneous Dirichlet/Neumann boundary conditions. Let $\{\varphi_k\}_{k\in\mathbb{N}}$ denotes an orthonormal basis of $L^2(\mathcal{O})$ satisfying
\[\begin{cases}
\mathcal{L}\varphi_k=\big(-\nabla\cdot (A(x)\nabla\big)\varphi_k=\lambda_k\varphi_k &\mbox{ in }\mathcal{O},\\
\varphi_k = 0 &\mbox{ on }\partial\mathcal{O}.
\end{cases}\]
Then, the spectral non-local operator $\mathcal{L}^s_S$ $(0<s<1$)  is defined as
\[\forall u\in H^s(\mathcal{O}),\quad\mathcal{L}^s_S u =\sum_{k=1}^\infty \lambda_k^s\langle \varphi_k,u\rangle_{2}\, \varphi_k \quad\mbox{ in }\mathcal{O}.\]
In \cite{caffarelli2016fractional} (also \cite{caffarelli2007extension}) the Caffarelli-Stinga result (see Proposition \ref{t24}, \ref{t16} for a statement of the result) was proved for this operator, thus achieving
a local problem posed on a semi-infinite cylinder $\mathcal{O}\times(0,\infty)$, whose Dirichlet-Neumann map defines the operator $\mathcal{L}^s_S$. Since our method to prove Theorem \ref{t9} relies on the analysis of the extended local problem in $\mathbb{R}^{n+1}_{+}$ (as shown in Section \ref{sect3}), we say the above theorem holds true for the homogenization of the spectral non-local operators $\{(\mathcal{L}^s_S)_\varepsilon\}_{\varepsilon>0}$ in $\mathcal{O}$.
\end{rem}

Let us now introduce our second problem to be considered and state the second main result.  
To this end, we define a sequence of any closed subsets $\{T_\eps\}_{\eps>0}\subset\R^n$,  which are called holes, and we take the  perforated domain $\mathcal{O}_\eps$ simply defined as
follows:\begin{equation}\label{t98} \mathcal{O}_\eps :=\mathcal{O}\setminus\underset{0<\delta\leq\eps}{\cup}\ T_\delta,
\end{equation}
with the condition on Lebesgue measure: 
\begin{equation}\label{t99} \lim_{\eps\to 0} |\mathcal{O}\setminus \mathcal{O}_\eps|=0.
\end{equation}
For $s\in(0,1) $ and for each $\eps>0$, let $u_\eps\in H^s(\mathbb{R}^n)$ be the solution of the following non-local Dirichlet problem in a perforated domain:
\begin{equation}\label{t10001}
\begin{cases}
(-\Delta)^s u_\eps=f  &\mbox{ in \ensuremath{\mathcal{O}_\eps}},
\\
u_\eps=g  &\mbox{ in }\mathbb{R}^n\setminus\mathcal{O}_\eps,
\end{cases}
\end{equation}
for some $f\in\widetilde{H}^{s}(\mathcal{O})^{*}$ and $g\in H^{s}(\mathbb{R}^{n})$. 

Motivated from the first result (cf. Theorem \ref{t9}) we allow such perforated domains $\mathcal{O}_\eps$ where the following hypothesis are satisfied.  Let us assume that there exist a sequence of functions $\{w_\eps\}_{\eps>0}$ such that: 
\begin{enumerate}
\item[(H1)] $w_\eps\in H^1(\mathcal{O})$;
\item[(H2)] $w_\eps=0$ on the holes $\underset{0<\delta\leq\eps}{\cup}T_\delta$;
\item[(H3)] $w_\eps \rightharpoonup 1$ weakly in $H^1(\mathcal{O})$.
\end{enumerate}

We now state our second main result and we show that the
commonly referred to
as "the `strange term'" in the literature  does not appear  in the homogenized
problem associated with  the fractional Laplace operator 
in a  perforated domain.
\begin{thm}\label{t63}
Let $s\in(0,1)$, $\mathcal{O}\subset\mathbb{R}^n$ be a bounded domain with
sufficiently smooth boundary and $\{\mathcal{O}_\eps\}_{\eps>0}$ be defined by \eqref{t98} with \eqref{t99}. We assume the above hypothesis (H1)-(H3) on $\{\mathcal{O}_\eps\}_{\eps>0}$.
For each $\eps>0$, let $u_\eps\in H^s(\mathbb{R}^n)$ be the solution of problem  \eqref{t10001} for given $f\in \widetilde{H}^{s}(\mathcal{O})^{*}$ and $g\in H^s(\mathbb{R}^n)$.
 Then, as $\eps\to 0,$ up to a subsequence, we have 
\[u_\eps \rightharpoonup u \quad \mbox{ weakly in }H^s(\mathbb{R}^n),\]
where the limit $u\in H^s(\mathbb{R}^n)$ can be characterized as the unique solution of the following homogenized problem:
\begin{equation}\label{t43}\begin{cases}
(-\Delta)^su  = f \quad&\mbox{in }\mathcal{O},
\\
u=g\quad&\mbox{in }\mathbb{R}^n\setminus\mathcal{O}.
\end{cases}
\end{equation}
\end{thm}
\begin{rem}
Since $\mathcal{L}^sw\to \mathcal{L}w$ in $L^2(\mathcal{O})$, as $s\to 1^{-}$, for $w\in H^2(\mathcal{O})$ (see \cite{grigoryan2009heat,EGV}), then from our result Theorem \ref{t9} we can essentially claim that the homogenization process, as $\eps\to 0$, is stable, under the limiting approach as $s\to 1^{-}$, that is, both of these limit operations, as $\eps\to 0$ and as $s\to 1^{-}$, are interchangeable. However, in  Theorem \ref{t63}, we find out  it is not the case. Both limiting processes, as $\eps\to 0$ and as $s\to 1^{-}$, may not be always interchangeable  because, in the local case, depending on the estimated size of the tiny holes $\{T_\eps\}_{\eps>0}$, one might end up having some nonzero zeroth order extra term (say $\mu(x)$) commonly referred to as a ``strange term''  with the Laplacian in the homogenized operator as $-\Delta+\mu$ (see \cite[Chapter 4]{MT}, \cite{DRL}).  
\end{rem}
The above homogenization results are new in the non-local settings and also help to provide a certain classification of perforated and non-perforated domains with respect to the fractional power of an elliptic operator. In both Theorem \ref{t9} and Theorem \ref{t63} we do not assume any periodicity or scaling  conditions, neither on the sequence $\{A_\eps(x)\}_{\eps>0}$ in the non-perforated case, nor on the sequence of perforated domains $\{\mathcal{O}_\eps\}_{\eps>0}$ respectively, in order to study the homogenization process.  

Let us now review few known studies of homogenization problems in non-local settings. The known cases are mostly in some prototype of integro-differential operator.  For example, a non-local linear operator with a kernel of convolution type in periodic medium \cite{PZ}, concerning certain diffusion process with jumps have been considered. That is also known Feller process generated by an integro-differential operator  \cite{SN}.  Homogenization of a certain class of integro-differential equations with L\'{e}vy operators \cite{ARM}, including scaling limits for symmetric It\^{o}-L\'{e}vy processes in random medium \cite{RVV} has been studied. Additionally, homogenization of a large class of fully non-linear elliptic integro-differential equations in periodic medium can be found in \cite{SCH}. For example, one prototype of  such integro-differential operator under consideration in \cite{PZ} is 
$$ \mathcal{L}^\eps u_\eps= \eps^{-n-2}\lambda\Big(\frac{x}{\eps}\Big)\int_{\mathbb{R}^n}a\Big(\frac{x-y}{\eps}\Big)\mu\Big(\frac{y}{\eps}\Big)\big( u_\eps(y)-u_\eps(x)\big)\, dy,$$ 
where $\lambda,\mu$ are bounded positive periodic functions characterizing the properties of the
medium, and $a$ is the jump kernel being a symmetric positive integrable function. They obtain the limit operator as a local operator $\mathcal{L}=-\sum\limits_{i,j=1}^n\Theta_{ij}\frac{\partial^2 u}{\partial x_i\partial x_j}$, where the homogenized coefficient $\Theta_{ij}$ can be derived from $a,\lambda,\mu$. 
In \cite{SCH2}, the author considers the 
 stochastic homogenization for elliptic integro-differential equations modelling stationary ergodic random environments:  
\begin{equation*}
F_\eps\Big(u_\eps,\frac{x}{\eps},\omega\Big)=\underset{\alpha}{\mbox{inf}}\, \underset{\beta}{\mbox{sup}}\,\Big\{ f^{\alpha\beta}\Big(\frac{x}{\eps},\omega\Big) + \int_{\mathbb{R}^n} \Big( u_\eps(x+y)+u_\eps(x-y)-2u_\eps(x)\Big)K^{\alpha\beta}\Big(\frac{x}{\eps},y,\omega\Big)\, dy\Big\}.
\end{equation*} 
Under some suitable conditions over the kernel $K^{\alpha\beta}$, the author obtains the homogenized equation as the certain viscosity solution of a  non-local, elliptic, and translation invariant operator of the same form above. However, in this paper, we don't restrict ourselves in certain examples and rather move into considering classical non-local elliptic problems in a bounded domain including both perforated and non-perforated types and study the homogenization process.


The outline of the remaining paper is the following. Section
\ref{Sect2}
deals with the functional framework of the fractional non-local elliptic operators $\mathcal{L}^s$. In
Section \ref{sect3} we introduce an extension problem which characterize
this non-local operator $\mathcal{L}^s$.
In Section \ref{sect4} we give the proof of our first main result. Finally,
Section \ref{sect5} focuses on the homogenization process of the fractional Laplace operator in perforated domains.

\section{Functional framework of the fractional non-local elliptic operator}
\label{Sect2}

Let us consider $\{\mathcal{L}_\eps\}_{\eps>0}$ a sequence of linear second order partial differential operator of the divergence form defined in the entire space $\mathbb{R}^n$ as
follows:
\begin{equation}\label{t39} 
\mathcal{L}_\eps:= -\nabla\cdot\left(A_\eps(x)\nabla\right),
\end{equation}
where  $\{A_\eps(x)\}_{\eps>0}=\{(a^{ij}(\frac x\eps))\}_{\eps>0}$, $x\in\mathbb{R}^{n}$, is a sequence of  $n\times n$
symmetric matrices satisfying the uniform ellipticity conditions \eqref{t51}.
We are going to study the  sequence of operators $\{\mathcal{L}_\eps^s\}_{\eps>0}$, with $0<s<1,$ \begin{equation*} 
\mathcal{L}_\eps^s:= (-\nabla\cdot\left(A_\eps(x)\nabla\right))^s,
\end{equation*}
defined over the entire space $\mathbb{R}^n$ and which will be completely defined in the sequel. 

Let us now consider the following  differential equation associated with this operator in the bounded domain $\mathcal{O}$:  
$$\mathcal{L}^s_\eps u_\eps =f \quad \mbox{ in }\mathcal{O},$$ for some suitable $f$. Next, in order to have a well-posed Dirichlet problem, we assume some exterior boundary condition as follows: 
$$u_\eps=g \quad \mbox{ in } \mathbb{R}^n\setminus\mathcal{O},$$ for some suitable $g$. Thus, the homogenization problem that we study is the following Dirichlet problem: 
\begin{equation*}
\begin{array}{rll}
\mathcal{L}_\eps^{s}u_\eps&=f  \quad\mbox{ in \ensuremath{\mathcal{O}}},
\\
u_\eps&= g \quad\mbox{ in }\mathbb{R}^n\setminus\mathcal{O}. 
\end{array}
\end{equation*}

Let us denote by $\mathcal{L}$ a second order linear
elliptic operator in the divergence form 
\begin{equation}
\mathcal{L}:=-\nabla\cdot(A(x)\nabla),\label{t4}
\end{equation}
which is defined in the entire space $\mathbb{R}^{n}$ for $n\geq2$,
where $A(x)=(a_{ij}(x))_{i,j}$, $x\in\mathbb{R}^{n}$ is an $n\times n$
symmetric matrix satisfying the symmetry and ellipticity conditions \eqref{t51}. We also assume that the variable coefficients of $\mathcal{L}$ are enough regular, precisely 
\begin{equation}
a^{ij}=a^{ji}\in C^{2}(\R^n),\quad1\le i,j\le n.\label{t35}
\end{equation}
It is well known that the operator $\mathcal{L}$ together with the domain 
\begin{equation}
\mathrm{Dom}(\mathcal{L})=H^{2}(\mathbb{R}^{n})\label{t37}
\end{equation}
is the maximal extension such that $\mathcal{L}$ is self-adjoint
and densely defined in $L^{2}(\mathbb{R}^{n})$ (see, for instance, \cite{grigoryan2009heat}).

\subsection{Fractional Sobolev spaces and non-local elliptic fractional differential operator }
\label{t14}

In this subsection, we will introduce the variable
coefficients fractional non-local operator $\mathcal{L}^{s}=(-\nabla\cdot(A(x)\nabla))^{s}$.
Let us note that, for $A(x)$ being an identity matrix, the operator $\mathcal{L}^{s}$
becomes the well-known fractional Laplace operator $(-\Delta)^{s}$, which has been widely studied in
papers \cite{caffarelli2007extension,caffarelli2016fractional,ruland2015unique,seeleycomplex} and the references therein.

In this paper, we denote by $C$
a general constant that may change in each occurrence and which will depend on the parameters involved. Wherever it is necessary, we are going to point out the dependence of $C$ on the parameters. Moreover, $\Gamma$ stands for the Gamma
function in the rest of the paper. 

Let us restrict our attention to the case $0 < s < 1$. In this interval, we have that $\Gamma(-s):= \displaystyle \frac{\Gamma(1-s)}{-s}< 0$.
\subsubsection*{Spectral
approach of non-local elliptic fractional differential operator}
We begin by defining the fractional operator $\mathcal{L}^{s}$ with 
$s\in(0,1)$, via the spectral characterization of $\mathcal{L}$ (for more details, see \cite{grigoryan2009heat,pazy2012semigroups,rudin1991functional,stinga2010extension}). Suppose that $\mathcal{L}$ is a linear second order differential
self-adjoint operator which is nonnegative and densely defined on
$L^{2}(\mathbb{R}^{n})$ for $n\geq2$. There is a unique resolution
$E$ of the identity, supported on the spectrum of $\mathcal{L}$
which is a subset of $[0,\infty)$, such that 
\[I=\int_{0}^{\infty}dE(\lambda)\]
and
\[
\mathcal{L}:=\int_{0}^{\infty}\lambda\, dE(\lambda),
\]
that is, 
\begin{equation}\label{t78}
\langle\mathcal{L}f,g\rangle_{L^{2}(\mathbb{R}^{n})}:=\int_{0}^{\infty}\lambda\, dE_{f,g}(\lambda),\ f\in\mbox{Dom}(\mathcal{L}),g\in L^{2}(\mathbb{R}^{n}),
\end{equation}
where $dE_{f,g}(\lambda)$ is a regular Borel complex measure of bounded
variation concentrated on the spectrum of $\mathcal{L}$, with \[
dE_{f,g}|_{|(0,\infty)}\leq\|f\|_{L^{2}(\mathbb{R}^{n})}\|g\|_{L^{2}(\mathbb{R}^{n})}.
\]
The norm  $\|\mathcal{L}f\|_{L^2(\mathbb{R}^n)}$, $f\in \mbox{Dom}(\mathcal{L})$, is  defined as follows:
\[ \|\mathcal{L}f\|^2_{L^2(\mathbb{R}^n)} := \int_{0}^{\infty}|\lambda|^2\, dE_{f,f}(\lambda).\]
If $\phi(\lambda)$ is a real measurable function defined on $[0,\infty)$,
then the operator $\phi(\mathcal{L})$ is formally given  by 
\[
\phi(\mathcal{L})=\int_{0}^{\infty}\phi(\lambda)dE(\lambda).
\]
That is, $\phi(\mathcal{L})$ is the operator with the domain 
\begin{equation}\label{t59}
\mbox{Dom}(\phi(\mathcal{L}))=\left\{ f\in L^{2}(\mathbb{R}^{n}):\int_{0}^{\infty}|\phi(\lambda)|^{2}dE_{f,f}(\lambda)<\infty\right\} ,
\end{equation}
defined by 
\begin{equation}\label{t85}
\left\langle \phi(\mathcal{L})f,g\right\rangle _{L^{2}(\mathbb{R}^{n})}=\int_{0}^{\infty}\phi(\lambda)\, dE_{f,g}(\lambda)
\end{equation}
and 
\begin{equation}\label{t82} \|\phi(\mathcal{L})f\|^2_{L^2(\mathbb{R}^n)}=\int_0^\infty |\phi(\lambda)|^2\, dE_{f,f}(\lambda).\end{equation}
Following this construction, we can define the fractional operators $\mathcal{L}^{s}$, $s\in(0,1)$, with the domain $\mbox{Dom}(\mathcal{L}^s)\, \supset\mbox{Dom}(\mathcal{L})$, as follows:
\begin{equation}
\mathcal{L}^{s}=\int_{0}^{\infty}\lambda^{s}~dE(\lambda)=\frac{1}{\Gamma(-s)}\int_{0}^{\infty}\left(e^{-t\mathcal{L}}-\mbox{Id}\right)~\frac{dt}{t^{1+s}}.
\label{t1}
\end{equation}
Here, $e^{-t\mathcal{L}}$
$(t\geq0)$ is the heat-diffusion semigroup generated by $\mathcal{L}$, 
with the domain $L^{2}(\mathbb{R}^{n})$, defined by 
\[e^{-t\mathcal{L}}=\int_{0}^{\infty}e^{-t\lambda}~dE(\lambda),\]
which enjoys the contraction property in $L^{2}(\mathbb{R}^{n})$, that is, $$\|e^{-t\mathcal{L}}f\|\leq\|f\|_{L^{2}(\mathbb{R}^{n})}.$$
Note that, for $f\in Dom(\mathcal{L}^s)\cap Dom(\mathcal{L}^{s/2})$, from \eqref{t85} it follows
\begin{equation}\label{t13}
\left\langle \mathcal{L}^sf,f\right\rangle _{L^{2}(\mathbb{R}^{n})}=\int_{0}^{\infty}\lambda^s\, dE_{f,f}(\lambda)= \|\mathcal{L}^{s/2}f\|^2_{L^2(\mathbb{R}^n)}.
\end{equation}
Moreover, for $f,g\in Dom(\mathcal{L}^s)\cap Dom(\mathcal{L}^{s/2})$, we have 
\begin{equation}
\label{t11}
\left\langle \mathcal{L}^sf,g\right\rangle _{L^{2}(\mathbb{R}^{n})}=\left\langle f,\mathcal{L}^sg\right\rangle _{L^{2}(\mathbb{R}^{n})}=\int_{0}^{\infty}\lambda^s\, dE_{f,g}(\lambda)= \left\langle \mathcal{L}^{s/2}f,\mathcal{L}^{s/2}g\right\rangle_{L^2(\mathbb{R}^n)},
\end{equation}
where we have used for  $f\in Dom(\mathcal{L}^{s/2})$ and $h \in Dom(\mathcal{L}^{s/2})$ that
\begin{equation}
\label{t2} \left\langle \mathcal{L}^{s/2}f,h\right\rangle_{L^2(\mathbb{R}^n)} = 
\int_{0}^{\infty}\lambda^{s/2}\, dE_{f,h}(\lambda).
\end{equation}
 Taking $h=\mathcal{L}^{s/2}g$ with $g\in Dom(\mathcal{L}^{s/2})$,  we deduce that  
\[dE_{f,h}=\lambda^{s/2}dE_{f,g}.\]

\subsubsection*{Kernel representation of the operator $\mathcal{L}^{s}$}

Let us write the definition given in \eqref{t1} for any $v\in Dom(\mathcal{L}^{s})$: 
\begin{equation}
\mathcal{L}^{s}v=\frac{1}{\Gamma(-s)}\int_{0}^{\infty}\left(e^{-t\mathcal{L}}v(x)-v(x)\right)\dfrac{dt}{t^{1+s}}.\label{t70}
\end{equation}

We introduce the distributional heat kernel $W_{t}(x,z)$ of
$\mathcal{L}$ satisfying: for any $\varphi,\mbox{ }\psi\in H^{s}(\mathbb{R}^{n})$,
\begin{equation}\label{t69}
(e^{-t\mathcal{L}}\varphi,\psi)_{\mathbb{R}^{n}}=\int_{\mathbb{R}^{n}}\int_{\mathbb{R}^{n}}W_{t}(x,z)\varphi(z)\psi(x)dzdx=(\varphi,e^{-t\mathcal{L}}\psi)_{\mathbb{R}^{n}},\mbox{ }t\geq0.
\end{equation}
Since $A(x)$ satisfies \eqref{t51} in $\mathbb{R}^n$, using \cite{Aronson}, it follows that, for some positive constants $c_1,c_2,c_3,c_4$ depending on ellipticity and boundness of $A$ and $n$, we have
\begin{equation}\label{t49} c_1\displaystyle\frac{e^{-\frac{|x-z|^2}{c_2}t}}{t^{n/2}}\leq W_t(x,z) \leq c_3\displaystyle\frac{e^{-\frac{|x-z|^2}{c_4}t}}{t^{n/2}}.\end{equation}
Let us now define the kernel of the heat semi-group $e^{-t\mathcal{L}}$ by 
\begin{equation}\label{t76}
\mathcal{K}^s(x,z)=\dfrac{1}{2|\Gamma(-s)|}\int_{0}^{\infty}W_{t}(x,z)\dfrac{dt}{t^{1+s}}.
\end{equation}
Since $e^{-t\mathcal{L}}$ is symmetric, we get  $\mathcal{K}^s(x,z)=\mathcal{K}^s(z,x)$
for any $x,z\in\mathbb{R}^{n}$, 
then from \cite[Theorem 2.4]{caffarelli2016fractional} it follows  that: for all  $v,w\in Dom(\mathcal{L}^{s}),$ 
\begin{equation}
(\mathcal{L}^{s}v,w)_{\mathbb{R}^{n}}=\int_{\mathbb{R}^{n}}\int_{\mathbb{R}^{n}}(v(x)-v(z))(w(x)-w(z))\mathcal{K}^s(x,z)\,dxdz.\label{t75}
\end{equation}
Furthermore, a direct computation and using  estimate \eqref{t49} on $W_t$, one can prove  that the kernel $\mathcal{K}^s$
enjoys the following pointwise estimate:
\begin{equation}
c_1\frac{\Gamma(\frac{n}{2}+s)}{2|\Gamma(-s)|}c_2^{\frac{n}{2} +s}\, \dfrac{1}{|x-z|^{n+2s}}\leq\mathcal{K}^s(x,z)\leq c_3\frac{\Gamma(\frac{n}{2}+s)}{2|\Gamma(-s)|}c_4^{\frac{n}{2} +s}\,\dfrac{1}{|x-z|^{n+2s}},\label{t86}
\end{equation}
where the constants $c_1,c_2,c_3,c_4$ appear in \eqref{t49} and are dependent on the ellipticity and boundness of $A$ and on $n$.  

We may also write for
 $v\in Dom(\mathcal{L}^{s})$ (for more details, see \cite{ghoshlinxiao}):
\begin{equation}
\mathcal{L}^{s}v(x)=\mbox{P.V.}\int_{\mathbb{R}^{n}}(v(x)-v(z))\mathcal{K}^s(x,z)\,dz,
\label{t45}
\end{equation}
where $\mbox{P.V.}$ stands for the standard principal value operator.
\subsubsection*{Sobolev spaces}
Let $H^{s}(\mathbb{R}^{n})=W^{s,2}(\mathbb{R}^{n})$ for $s\in \mathbb{R}$ the standard Sobolev space with the norm 
\[
\|u\|_{H^{s}(\mathbb{R}^{n})}=\|\left\langle D\right\rangle ^{s}u\|_{L^{2}(\mathbb{R}^{n})},
\]
where $\left\langle \xi\right\rangle =(1+|\xi|^{2})^{\frac{1}{2}}$.
Let $m(\xi)$ be an arbitrary $C^{\infty}$-smooth polynomial in $\xi$,
and the notation $m(D)u=\mathscr{F}^{-1}\{m(\xi)\hat{u}(\xi)\}$ stands
for the Fourier multipliers and $\mathscr{F}$ is the classical Fourier transform
given by 
\[
\widehat{u}(\xi)=\mathscr{F}u(\xi)=\int_{\mathbb{R}^{n}}e^{-ix\cdot\xi}u(x)\,dx.
\]
We may and shall consider the following $H^s(\mathbb{R}^n)$-norm,  for $s\in \mathbb{R}^{+}$:
\begin{equation}
\|u\|^2_{H^s(\mathbb{R}^n)}= \|u\|^2_{L^2(\mathbb{R}^n)} + \|(-\Delta)^{s/2}u\|^2_{L^2(\mathbb{R}^n)}.
\end{equation}
Let us observe that the semi-norm $\|(-\Delta)^{s/2}u\|^2_{L^2(\mathbb{R}^n)}$ is expressed as follows: 
\[ \|(-\Delta)^{s/2}u\|^2_{L^2(\mathbb{R}^n)}=\left((-\Delta)^su, u\right)_{\mathbb{R}^n},\]
where, for $s\in (0,1)$,
\begin{equation}\label{t66} (-\Delta)^su(x) = c_{n,s}\, \mbox{P.V.}\int_{\mathbb{R}^n}\frac{u(x)-u(y)}{|x-y|^{n+2s}}\, dy\end{equation}
and 
\begin{equation}
\label{t50}
c_{n,s}= \frac{\Gamma(\frac{n}{2}+s)}{|\Gamma(-s)|}\,\frac{4^s}{\pi^{n/2}}.
\end{equation}

The space  $H^{s}(\mathcal{O})$, 
with $\mathcal{O}\subset\mathbb{R}^{n}$ being an arbitrary open set,
is equipped with the following norm (see \cite[Chapter 3]{mclean2000strongly}):
\[
\|u\|_{H^{s}(\mathcal{O})}:=\inf\left\{ \|w\|_{H^{s}(\mathbb{R}^{n})}:w\in H^{s}(\mathbb{R}^{n})\mbox{ and }w|_{\mathcal{O}}=u\right\} .
\]
Furthermore, by taking $\mathcal{C}\subset\mathbb{R}^{n}$  a closed
set such that int$(\mathcal{C})\neq\emptyset$, we can define 
\[
H_{\mathcal{C}}^{s}=H_{\mathcal{C}}^{s}(\mathbb{R}^{n})=\left\{ u\in H^{s}(\mathbb{R}^{n}):\mbox{ }\mbox{supp}(u)\subset\mathcal{C}\right\} .
\]
If $\mathcal{O}$ is a Lipschitz domain, then we have the following
space identification (for more details, see \cite{mclean2000strongly,triebel2002function}):
for $s\in\mathbb{R}$, 
\begin{align}
\label{def:Hstilde} & \widetilde{H}^{s}(\mathcal{O})=H_{\overline{\mathcal{O}}}^{s}(\mathbb{R}^{n}),
 \\
\label{def:Hstilde1} & \widetilde{H}^{s}(\mathcal{O})^*=H_{\overline{\mathcal{O}}}^{s}(\mathbb{R}^{n})^{*}=H^{-s}(\mathcal{O})\mbox{ and }H^{s}(\mathcal{O})^{*}=H_{\overline{\mathcal{O}}}^{-s}(\mathbb{R}^{n})
\end{align}
and, for $s\in(-\frac{1}{2},\frac{1}{2})$, 
\[
H^{s}(\mathcal{O})=H_{\overline{\mathcal{O}}}^{s}(\mathbb{R}^{n})=H_{0}^{s}(\mathcal{O}).
\]

\subsection{Dirichlet problem for $\mathcal{L}^{s}$}

We consider the following Dirichlet problem for the
non-local operator $\mathcal{L}^{s}$ 
\begin{equation}
\begin{cases}
\mathcal{L}^{s}u=f & \mbox{ in \ensuremath{\mathcal{O}}},
\\
u=g & \mbox{ in }\mathcal{O}_{e}=\R^n\setminus \mathcal O,
\end{cases}\label{t81}
\end{equation}
with $f\in\widetilde{H}^{s}(\mathcal{O})^{*}$ and $g\in H^{s}(\mathbb{R}^{n})$.

Let us first observe that for any $v\in H^{s}(\mathbb{R}^{n})$ with $s\in (0,1)$, $\mathcal{L}^{s}v$
can be defined as a distribution in $H^{-s}(\mathbb{R}^{n})$ by \eqref{t75}
as follows: 
\begin{multline}
\left|(\mathcal{L}^{s}v,w)_{\mathbb{R}^{n}}\right|  =  \left|\int_{\mathbb{R}^{n}}\int_{\mathbb{R}^{n}}(v(x)-v(z))(w(x)-w(z))\mathcal{K}^s(x,z)dxdz\right| \\
  \leq  \left(\int_{\mathbb{R}^{n}}\int_{\mathbb{R}^{n}}|v(x)-v(z)|^{2}\mathcal{K}^s(x,z)dxdz\right)^{\frac{1}{2}}\cdot\left(\int_{\mathbb{R}^{n}}\int_{\mathbb{R}^{n}}|w(x)-w(z)|^{2}\mathcal{K}^s(x,z)dxdz\right)^{\frac{1}{2}} \\
 \leq  C\|v\|_{H^{s}(\mathbb{R}^{n})}\|w\|_{H^{s}(\mathbb{R}^{n})},\label{t23}
\end{multline}
for any $w\in H^{s}(\mathbb{R}^{n})$. Here, we have used that $\mathcal{K}^s(x,z)\geq0$
for all $x\neq z$ and also the estimate \eqref{t86}.

We then consider the following associated bilinear form of the above non-local
problem \eqref{t81}: for any $v,\mbox{ }w\in H^{s}(\mathbb{R}^{n})$,
\begin{equation}
\mathcal{B}^s(v,w):=\int_{\mathbb{R}^{n}}\int_{\mathbb{R}^{n}}(v(x)-v(z))(w(x)-w(z))\mathcal{K}^s(x,z)\,dx\,dz.\label{t18}
\end{equation}
It is easy to see from estimate \eqref{t23} that the above bilinear
form $\mathcal{B}^s(\cdot,\cdot)$ is well-defined  in $H^{s}(\mathbb{R}^{n})\times H^{s}(\mathbb{R}^{n})$,
i.e., 
\begin{equation}\label{t90}
\left|\mathcal{B}^s(v,w)\right|\leq C\|v\|_{H^{s}(\mathbb{R}^{n})}\|w\|_{H^{s}(\mathbb{R}^{n})}.
\end{equation}
We  note that, following \eqref{t11}, the bilinear form  $\mathcal{B}^s$ can be also expressed as
follows:\[\mathcal{B}^s(v,w)= \langle \mathcal{L}^{s/2}v,\mathcal{L}^{s/2}w\rangle_{L^2(\mathbb{R}^n)}\quad\forall v,w\in H^s(\mathbb{R}^n).\]
Thus, we have the following existence result (for the complete proof, we refer the reader to \cite{ghoshlinxiao}):


\begin{prop}\label{t94}
Let $\mathcal{O}\subset\mathbb{R}^{n}$ as mentioned above, and $\mathcal{B}^s$
is a bilinear form defined in \eqref{t18}, then there is a solution
$u\in H^{s}(\mathbb{R}^{n})$ such that 
\begin{equation}
\mathcal{B}^s(u,w)=\left\langle f,w\right\rangle \mbox{ for any }w\in\widetilde{H}^{s}(\mathcal{O})\mbox{ with }u-g\in\widetilde{H}^{s}(\mathcal{O}),\label{t10}
\end{equation}
for any $f\in\widetilde{H}^{s}(\mathcal{O})^{*}$ and $g\in H^{s}(\mathbb{R}^{n})$,
where $\left\langle \cdot,\cdot\right\rangle $ stands for the duality
pairing between $(\widetilde{H}^{s})^{*}$ and $\widetilde{H}^{s}$.

Since $0$ is not the eigenvalue of the problem
\begin{equation*}
\left\{
\begin{array}{lll}
\mathcal{L}^s w=0\quad \mbox{ in }\mathcal{O}, 
\\[1mm]
w=0 \quad \mbox{ in }\mathcal{O}_e,
\end{array}
\right.
\end{equation*}
 the above solution $u\in H^s(\mathbb{R}^n)$ is unique. 

In addition, we have the following estimate: 
\begin{equation}
\|u\|_{H^{s}(\mathbb{R}^{n})}\leq C\big(\|f\|_{\widetilde{H}^{s}(\mathcal{O})^{*}}+\|g\|_{H^{s}(\mathbb{R}^{n})}\big),\label{t92}
\end{equation}
for some constant $C>0$ independent of $f$ and $g$ and depending on the ellipticity and boundedness of $A$ (see \eqref{t49}) and on the dimension $n$.
\end{prop}
\begin{rem}
The solution $u\in H^s(\mathbb{R}^n)$ of problem \eqref{t81} does not depend on the value of $g\in H^s(\mathbb{R}^n)$ on $\mathcal{O}$, it only depends  on $g|_{\mathcal{O}_e}$.  Let $g_{1},g_{2}\in H^{s}(\mathbb{R}^{n})$
be such that $g_{1}-g_{2}\in\widetilde{H}^{s}(\mathcal{O})=H_{\overline{\mathcal{O}}}^{s}$.
Denote by $u_{j}\in H^{s}(\mathbb{R}^{n})$ the solution of \eqref{t81}
with the Dirichlet data $g_{j}$ for each $j=1,2$. It is observed
that 
\[
\widetilde{u}:=u_{1}-u_{2}=(u_{1}-g_{1})-(u_{2}-g_{2})+(g_{1}-g_{2})\in\widetilde{H}^{s}(\mathcal{O})
\]
and $\mathcal{B}_{q}(\widetilde{u},v)=0$ for any $v\in\widetilde{H}^{s}(\mathcal{O})$.
Thus, by  unicity of solution of \eqref{t81}
with $g=0$, one has $\widetilde{u}=0$. Therefore, one can actually
consider the non-local problem \eqref{t81}
with Dirichlet data in the quotient space 
\begin{equation}
X:=H^{s}(\mathbb{R}^{n})/H_{\overline{\mathcal{O}}}^{s}\cong H^{s}(\mathcal{O}_{e}),\label{t48}
\end{equation}
provided that $\mathcal{O}$ is Lipschitz. 
\end{rem}
\begin{rem}[Flux estimate] 
Using  \eqref{t10}, it follows that 
\[\mathcal{B}^s(u,u-g) =\langle f, u-g\rangle,\]
for $f\in\widetilde{H}^{s}(\mathcal{O})^{*}$ and $g\in H^{s}(\mathbb{R}^{n})$. 
Then, we get
\[
\| \mathcal{L}^{s/2}u\|^2_{L^2(\mathbb{R}^n)}- \langle \mathcal{L}^{s/2}u,\mathcal{L}^{s/2}g\rangle_{L^2(\mathbb{R}^n)} =\langle f,u-g\rangle, 
\]
which implies that \[
\frac{1}{2}\| \mathcal{L}^{s/2}u\|^2_{L^2(\mathbb{R}^n)}\leq \frac{1}{2}\|\mathcal{L}^{s/2}g\|^2_{L^2(\mathbb{R}^n)} +\|f\|_{\widetilde{H}^{s}(\mathcal{O})^{*}}\big(\|u\|_{H^s(\mathbb{R}^n)}+\|g\|_{H^{s}(\mathbb{R}^{n})}\big). 
\]
Then, by using  $H^s(\mathbb{R}^n)$-estimate \eqref{t92} in the right hand side, we simply obtain
\begin{equation}\label{t22}
\| \mathcal{L}^{s/2}u\|_{L^2(\mathbb{R}^n)}\leq C\big( \|f\|_{\widetilde{H}^{s}(\mathcal{O})^{*}}+\|g\|_{H^{s}(\mathbb{R}^{n})}\big), 
\end{equation}
for some constant $C>0$ independent of $f$ and $g$ and depending on the ellipticity and boundness of $A$ and  on $n$.
\end{rem}
\subsection{Limit analysis of $\{u_\eps\}_{\eps>0}$ as $\eps \to 0$}

  We consider following sequence of 
non-local operators $\{\mathcal{L}^{s}_\eps\}_{\eps>0}= \big\{\big(-\nabla\cdot (A_\eps(x)\nabla)\big)^{s}\big\}_{\eps>0}$ introduced similar to the operator $\mathcal L^s$, with the sequence $\{A_\eps(x)\}_{\eps>0}$ satisfying the conditions \eqref{t51} and regularity condition \eqref{t35}.  For each $\eps>0$, let $u_\eps\in H^s(\mathbb{R}^n)$ solving 
\begin{equation}
\begin{cases}
\mathcal{L}_\eps^{s}u_\eps=f & \mbox{ in \ensuremath{\mathcal{O}}},\\
u_\eps=g & \mbox{ in }\mathcal{O}_{e},
\end{cases}\label{t79}
\end{equation}
for $f\in \widetilde{H}^{s}(\mathcal{O})^{*}$ and $g\in H^s(\mathbb{R}^n)$ and satisfying the stability and flux estimates:
\begin{equation}
\|u_\eps\|_{H^{s}(\mathbb{R}^{n})}\leq C\big(\|f\|_{\widetilde{H}^{s}(\mathcal{O})^{*}}+\|g\|_{H^{s}(\mathbb{R}^{n})}\big)\label{t62}
\end{equation}
and\begin{equation}
\|\mathcal{L}^{s/2}_\eps u_\eps\|_{L^{2}(\mathbb{R}^{n})}\leq C\big(\|f\|_{\widetilde{H}^{s}(\mathcal{O})^{*}}+\|g\|_{H^{s}(\mathbb{R}^{n})}\big),\label{t29}
\end{equation}
for some constant $C>0$ independent of $f$ and $g$ and dependent on the uniform ellipticity and boundness of $A_\eps$ and on $n$. Thus, $C$ is also independent of $\eps>0$.  
Therefore,  the sequences $\{u_\eps\}_{\eps>0}$ and  $\{\mathcal{L}^{s/2}_\eps u_\eps\}_{\eps>0}$ remain bounded  in $H^s(\mathbb{R}^n)$ (see \eqref{t62}) and $L^2(\mathbb{R}^n)$, respectively (see \eqref{t29}). Hence, upto a subsequence still denoted by same $\{u_\eps\}_{\eps>0}$, we get 
\begin{equation}
\label{t31} 
u_\eps \rightharpoonup u \quad \mbox{ weakly in }H^s(\mathbb{R}^n)
\end{equation}
and 
\begin{equation}
\label{t32} \mathcal{L}^{s/2}_\eps u_\eps \rightharpoonup v \quad \mbox{ weakly in }L^2(\mathbb{R}^n).
\end{equation}
In the sequel,  our goal is to find the homogenized problem or the limit equation satisfied by $u\in H^s(\mathbb{R}^n)$, and also the relation between both weak limits $u$ and $v$. 

To this end, we will proceed by using the extension techniques for the non-local operators, where the extended operator becomes a local operator.  

\section{Extension problems for $\mathcal{L}^{s}$}
\label{sect3}

In this section, we introduce an extension problem, which characterize
the non-local operator $\mathcal{L}^{s}$. 

To this end, let $\mathbb{R}_{+}^{n+1}:=\left\{ (x,y):\,x\in\mathbb{R}^{n},y>0\right\} $ 
be the upper half space of $\mathbb{R}^{n+1}$ with
 its boundary $\partial\mathbb{R}_{+}^{n+1}:=\left\{ (x,0):\,x\in\mathbb{R}^{n}\right\} $. Let $\omega$ be an arbitrary $A_{2}$-Muckenhoupt
weight function (for more details, see \cite{fabes1982local,muckenhoupt1972weighted})
and we denote by $L^{2}(\mathbb{R}^{n+1}_{+},\omega)$ the weighted Sobolev space
containing all functions $U$ which are defined a.e. in $\mathbb{R}_{+}^{n+1}$
such that 
\[
\|U\|_{L^{2}(\mathbb{R}_{+}^{n+1},\omega)}:=\left(\int_{\mathbb{R}_{+}^{n+1}}\omega|U|^{2}dxdy\right)^{1/2}<\infty.
\]
We define 
\[
H^{1}(\mathbb{R}_{+}^{n+1},\omega):=\Big\{U\in L^{2}(\mathbb{R}_{+}^{n+1},\omega):\,\nabla_{x,y}U\in L^{2}(\mathbb{R}_{+}^{n+1},\omega)\Big\},
\]
where $\nabla_{x,y}:=(\nabla,\partial_{y})=(\nabla_{x},\partial_{y})$
is the total derivative in $\mathbb{R}^{n+1}_{+}$. In this work, the
weight function $\omega$ might be $y^{1-2s}$ (or $y^{2s-1}$) and it is known  that
$y^{1-2s}\in A_{2}$ for $s\in(0,1)$ (see \cite{kufner1987some}). It is easy to see that $L^{2}(\mathbb{R}_{+}^{n+1},\omega)$
and $H^{1}(\mathbb{R}_{+}^{n+1},\omega)$ are Banach spaces with respect to the
norms $\|\cdot\|_{L^{2}(\mathbb{R}_{+}^{n+1},\omega)}$ and 
\begin{equation}\label{t84}
\|U\|_{H^{1}(\mathbb{R}_{+}^{n+1},\omega)}:=\left(\|U\|_{L^{2}(\mathbb{R}_{+}^{n+1},\omega)}^{2}+\|\nabla_{x,y}U\|_{L^{2}(\mathbb{R}_{+}^{n+1},\omega)}^{2}\right)^{1/2},
\end{equation}
respectively. We shall also make use of the weighted Sobolev space
$H_{0}^{1}(\mathbb{R}_{+}^{n+1},\omega)$ which is the closure of $C_{0}^{\infty}(\mathbb{R}_{+}^{n+1})$
under the $H^{1}(\mathbb{R}_{+}^{n+1},\omega)$-norm.

We mention that the
fractional Sobolev space $H^{s}(\mathbb{R}^{n})$ can be obtained
as the trace space of the weighted Sobolev space $H^{1}(\mathbb{R}_{+}^{n+1},y^{1-2s})$,
for $s\in(0,1)$, (see \cite{tyulenev2014description}), that is,
\begin{equation}\label{t6}Tr: H^1(\mathbb{R}^{n+1}_{+},y^{1-2s})\to H^s(\mathbb{R}^n)\end{equation}
is continuous. This means that, for
a given $u\in H^{s}(\mathbb{R}^{n})$, there exists $U(x,y)\in H^{1}(\mathbb{R}_{+}^{n+1},y^{1-2s})$
such that $\lim_{y\to 0^{+}}U(x,y)=U(x,0)=u(x)\in H^{s}(\mathbb{R}^{n})$ with 
\begin{equation}
\|u\|_{H^{s}(\mathbb{R}^{n})}\leq C\|U\|_{H^{1}(\mathbb{R}_{+}^{n+1},y^{1-2s})}.\label{t27}
\end{equation}

It also follows that for any bounded open strip away from ${y=0}$, say \[D_{(a,b)}=\{ (x,y)\in \mathbb{R}^{n}\times (a,b):\, \mbox{ $0<a<y<b<\infty$}\},\]
we have: $U\in H^1(\mathbb{R}^{n+1}_{+},y^{1-2s})$ for $s\in (0,1),$ implies $U\in H^1(D_{(a,b)})$ and  also  \begin{equation}\label{t83} \|U\|_{H^1(D_{(a,b)})}\leq C_{a,b} \|U\|_{H^1(\mathbb{R}_{+}^{n+1},y^{1-2s})}.\end{equation}
This is simply a consequence of definition \eqref{t84}, since the weight $y^{1-2s}$ is smooth enough  and
positive in $\overline{D_{(a,b)}}.$

Let us now consider the following extension problem in $\mathbb{R}^{n+1}_{+}$:
\begin{equation}
\begin{cases}
-\mathcal{L}_{x}U+\displaystyle\frac{1-2s}{y}U_{y}+U_{yy}=0 & \mbox{ in }\mathbb{R}_{+}^{n+1},
\\[1ex]
U(\cdot,0)=u(\cdot) & \mbox{ on }\partial\mathbb{R}_{+}^{n+1}.
\end{cases}\label{t41}
\end{equation}
This extension problem is related to the non-local operator \eqref{t1}, 
where the non-local operator $\mathcal{L}^{s}$ has been regarded as
a Dirichlet-to-Neumann map of the above degenerate local problem \eqref{t41}.
For convenience, we construct an auxiliary matrix-valued function
$\widetilde{A}:\mathbb{R}^{n}\to\mathbb{R}^{(n+1)\times(n+1)}$ by
\begin{equation}\label{t34}
\widetilde{A}(x)=\left(\begin{array}{cc}
A(x) & 0\\
0 & 1
\end{array}\right).\end{equation}
We introduce the following degenerate local operator: 
\begin{equation}
\mathscr{L}_{\widetilde{A}}^{1-2s}=\nabla_{x,y}\cdot(y^{1-2s}\widetilde{A}(x)\nabla_{x,y}).\label{t33}
\end{equation}
It can be seen that $y^{-1+2s}\mathscr{L}_{\widetilde{A}}^{1-2s}$
is nothing else than the above degenerate local operator defined in \eqref{t41}, precisely by

\begin{equation}\label{t77}
\mathscr{L}_{\widetilde{A}}^{1-2s}=y^{1-2s}\Big\{ \nabla\cdot(A(x)\nabla)+\frac{1-2s}{y}\partial_{y}+\partial_{y}^{2}\Big\} .
\end{equation}

Let us now recall the following existence result of the above extension problem \eqref{t41}, which complete proof can be found in \cite{ghoshlinxiao}:
\begin{prop}
\label{t24}Let $s\in(0,1)$
and  $\widetilde{A}$ be given by \eqref{t34}, with $A(x)$
satisfying the elipticity condition \eqref{t51}. Then, for  given
$u\in H^{s}(\mathbb{R}^{n})$, there exists a 
 unique minimizer of the Dirichlet functional 
\[
\min_{\Psi\in H^{1}(\mathbb{R}_{+}^{n+1},y^{1-2s})}\left\{ \int_{\mathbb{R}_{+}^{n+1}}y^{1-2s}\widetilde{A}(x)\nabla_{x,y}\Psi\cdot\nabla_{x,y}\Psi dxdy:\,\Psi(x,0)=u(x)\right\},
\]
characterized as the unique weak solution $U\in H^{1}(\mathbb{R}_{+}^{n+1},y^{1-2s})$ solving the problem
\begin{equation}\label{t56}
\begin{cases}
\mathscr{L}_{\widetilde{A}}^{1-2s}U=0 & \mbox{ in }\mathbb{R}_{+}^{n+1},\\
U(\cdot,0)=u & \mbox{ in }\mathbb{R}^{n},
\end{cases}\end{equation}
and satisfying the following stability estimate 
\begin{equation}
\|U\|_{H^{1}(\mathbb{R}_{+}^{n+1},y^{1-2s})}\leq C\|u\|_{H^{s}(\mathbb{R}^{n})},\label{t25}
\end{equation}
for some $C>0$ independent of $u$ and $U$ and depending only on the ellipticity and boundness of $A$ and on $n$. 
\end{prop}
\begin{proof}
The proof could be find in the recent paper \cite{ghoshlinxiao}. For our own convenience, we  mention here the apriori estimate \eqref{t25} in order to show that the constant $C>0$ appearing in \eqref{t25} depends only on the ellipticity and boundness of $A$ and on $n$.  

Given $u\in H^{s}(\mathbb{R}^{n})$, there exists $U_{0}(x,y)\in H^{1}(\mathbb{R}_{+}^{n+1},y^{1-2s})$
such that $\lim_{y\to 0^{+}}U_{0}(x,y)=U_{0}(x,0)=u(x)$ and by using the right continuity of the inverse trace map, we assume $\|U_0\|_{H^{1}(\mathbb{R}_{+}^{n+1},y^{1-2s})}\leq C\|u\|_{H^{s}(\mathbb{R}^{n})}$, where the constant $C>0$ is independent of $U_0 \in H^{1}(\mathbb{R}_{+}^{n+1},y^{1-2s})$ and $u\in H^s(\mathbb{R}^n)$.

  Since $U\in H^{1}(\mathbb{R}_{+}^{n+1},y^{1-2s})$
is the weak solution of \eqref{t56}, let define $V:=U-U_{0}$. Then $V\in H^{1}(\mathbb{R}_{+}^{n+1},y^{1-2s})$ is the weak solution
of 
\begin{equation}\label{t93}
\begin{cases}
\nabla\cdot(y^{1-2s}\widetilde{A}\nabla V)=\nabla\cdot G & \mbox{ in }\mathbb{R}_{+}^{n+1},
\\
V(x,0)=0 & \mbox{ in }\mathbb{R}^{n},
\end{cases}
\end{equation}
where $G:=-y{}^{1-2s}\widetilde{A}(x)\nabla_{x,y}U_{0}$. It is easy
to see that $y^{2s-1}G\in L^{2}(\mathbb{R}_{+}^{n+1},y^{1-2s})$ and
\begin{eqnarray*}
\int_{\mathbb{R}_{+}^{n+1}}y^{1-2s}|y^{2s-1}G|^{2}dx\,dy & = & \int_{\mathbb{R}_{+}^{n+1}}y^{1-2s}\left|\widetilde{A}\nabla_{x,y}U_{0}\right|^{2}\, dx\,dy\\
 & \leq & C\int_{\mathbb{R}_{+}^{n+1}}y^{1-2s}|\nabla_{x,y}U_{0}|^{2}\,dx\,dy,
\end{eqnarray*}
for some constant $C>0$ which depends only on boundness of $A$. Then, by multiplying \eqref{t93} by $V\in H^1_0(\mathbb{R}^{n+1}_{+},y^{1-2s})$ and integrating by parts,  we get
\[
\|V\|_{H^{1}(\mathbb{R}_{+}^{n+1},y^{1-2s})}  \leq  C\|y^{-1+2s}G\|_{L^{2}(\mathbb{R}_{+}^{n+1},y^{1-2s})},\]
for some constant $C>0$ which depends only on the ellipticity of $A$.
Finally,
\begin{eqnarray*}
\|V\|_{H^{1}(\mathbb{R}_{+}^{n+1},y^{1-2s})}
 & \leq & C\|U_{0}\|_{H^{1}(\mathbb{R}_{+}^{n+1},y^{1-2s})}
  \leq  C\|u\|_{H^{s}(\mathbb{R}^{n})}\\
\mbox{or,}\qquad \|U\|_{H^{1}(\mathbb{R}_{+}^{n+1},y^{1-2s})}& \leq & C\|u\|_{H^{s}(\mathbb{R}^{n})},
\end{eqnarray*}
for some universal constant $C>0$ which depends only on the ellipticity and on the boundness of $A$.
\hfill\end{proof}
As a consequence, we observe that $y^{1-2s}\partial_{y}U$
converges to some function $h\in H^{-s}(\mathbb{R}^{n})$, as $y\to0$,
in $H^{-s}(\mathbb{R}^{n})$ defined as follows:
\begin{equation}
(h,\phi(x,0))_{H^{-s}(\mathbb{R}^{n})\times H^{s}(\mathbb{R}^{n})}=\int_{\mathbb{R}_{+}^{n+1}}y^{1-2s}\widetilde{A}(x)\nabla_{x,y}U\cdot\nabla_{x,y}\phi\, dx\,dy,\label{t42}
\end{equation}
for all $\phi\in H^{1}(\mathbb{R}_{+}^{n+1},y^{1-2s})$. In other
words, $U\in H^{1}(\mathbb{R}_{+}^{n+1},y^{1-2s})$ is the weak solution
of the following Neumann boundary value problem 
\begin{equation}
\begin{cases}
\nabla_{x,y}\cdot(y^{1-2s}\widetilde{A}(x)\nabla_{x,y}U)=0 & \mbox{ in }\mathbb{R}_{+}^{n+1},\\
\lim\limits_{y\to 0^{+}}y^{1-2s}\partial_{y}U=h & \mbox{ in }\mathbb{R}^{n}\times\{0\}.
\end{cases}\label{t15}
\end{equation}
The following result characterizes $\lim_{y\to 0^{+}}y^{1-2s}\partial_{y}U=h$,
as $d_{s}h=\mathcal{L}^{s}u$, for some constant $d_{s}$ depending
on $s$, which connects the non-local operator $\mathcal{L}^s$ and the extension problem:
\begin{prop}
\label{t16} Given $u\in H^{s}(\mathbb{R}^{n})$,
define 
\begin{equation}
U(x,y):=\int_{\mathbb{R}^{n}}P_{y}^{s}(x,z)u(z)\,dz,\label{t47}
\end{equation}
where $P_{y}^{s}$ is the Poisson kernel given by 
\begin{equation}
P_{y}^{s}(x,z)=\dfrac{y^{2s}}{4^{s}\Gamma(s)}\int_{0}^{\infty}e^{-\frac{y^{2}}{4t}}W_{t}(x,z)\dfrac{dt}{t^{1+s}},\quad x,z\in\mathbb{R}^{n},\,y>0,\label{t88}
\end{equation}
with the heat kernel $W_{t}(x,z)$ introduced in \eqref{t69}.
Then $U\in H^{1}(\mathbb{R}_{+}^{n+1},y^{1-2s})$  is the weak
solution of \eqref{t56} and 
\begin{equation}
\lim_{y\to0^{+}}\dfrac{U(\cdot,y)-U(\cdot,0)}{y^{2s}}=\frac{1}{2s}\lim_{y\to0+}y^{1-2s}\partial_{y}U(\cdot,y)=\frac{\Gamma(-s)}{4^{s}\Gamma(s)}\mathcal{L}^{s}u(\cdot),\label{t55}
\end{equation}
in $H^{-s}(\mathbb{R}^{n})$.\end{prop}
\begin{proof}
The proof can be found in \cite{stinga2010extension}, where the authors prove the
equality \eqref{t55} for $u\in\mathrm{Dom}(\mathcal{L}^{s})$, and recently, in \cite{ghoshlinxiao} the result has been extended for $u\in H^s(\mathbb{R}^n)$.
\hfill\end{proof}
\subsection{Limiting analysis of $\{U_\eps\}_{\eps>0}$ as $\eps\to 0$}
  We consider the following sequence of 
local operators:
\begin{align*} 
\big\{\mathscr{L}_{\widetilde{A}_\eps}^{1-2s}\big\}_{\eps>0}&=\big\{\nabla_{x,y}\cdot(y^{1-2s}\widetilde{A}_\eps(x)\nabla_{x,y})\big\}_{\eps>0}\\
&=\Big\{ y^{1-2s}\Big( \nabla\cdot(A_\eps(x)\nabla)+\frac{1-2s}{y}\partial_{y}+\partial_{y}^{2}\Big) \Big\}_{\eps>0}\end{align*} introduced in \eqref{t77}, with the sequence $\{A_\eps(x)\}_{\eps>0}$ satisfying the ellipticity and boundness conditions \eqref{t51} and regularity condition \eqref{t35}. For each $\eps>0$, let us consider $U_\eps\in H^1(\mathbb{R}^{n+1}_{+},y^{1-2s})$  the solution of the following problem:  
\begin{equation}\label{t38}
\begin{cases}
\mathscr{L}_{\widetilde{A}_\eps}^{1-2s}U_\eps=0 & \mbox{ in }\mathbb{R}_{+}^{n+1},\\
U_\eps(\cdot,0)=u_\eps(\cdot) & \mbox{ in }\mathbb{R}^{n},
\end{cases}\end{equation}
which satisfies the stability estimate
\begin{align}
\|U_\eps\|_{H^1(\mathbb{R}^{n+1}_{+},y^{1-2s})}\leq C\|u_\eps\|_{H^s(\mathbb{R}^n)}
\leq C\big(\|f\|_{\widetilde{H}^{s}(\mathcal{O})^{*}}+\|g\|_{H^{s}(\mathbb{R}^{n})}\big),
\label{t97}
\end{align}
for some constant $C>0$ dependent on $n$, on the uniform ellipticity and boundness of $A_\eps$ and independent of $\eps>0$.  
Due to the above estimate \eqref{t97}, the sequence $\{U_\eps\}_{\eps>0}$ remains bounded  in $H^1(\mathbb{R}^{n+1}_{+},y^{1-2s})$. Therefore, upto a subsequence still denoted by same $\{U_\eps\}_{\eps>0}$, the sequence weakly converges to some limit $U\in H^1(\mathbb{R}^{n+1}_{+},y^{1-2s})$, that is,
\begin{equation}
\label{t91} 
U_\eps \rightharpoonup U \quad \mbox{ weakly in }H^1(\mathbb{R}^{n+1}_{+},y^{1-2s}).
\end{equation}Consequently, by the continuity of the trace map $Tr: H^1(\mathbb{R}^{n+1}_{+},y^{1-2s})\to H^s(\mathbb{R}^n)$ (see \eqref{t27}), we have 
\[Tr(U_\eps)\rightharpoonup  Tr(U) \quad \mbox{ weakly in } H^s(\mathbb{R}^n).\] Since \eqref{t31} holds, we get that $Tr(U_\eps)=u_\eps$ weakly converges to $u\in H^s(\mathbb{R}^n)$, hence, by the uniqueness of the weak limit in $H^s(\mathbb{R}^n)$, we find
\begin{equation}\label{t89}  u(x)= \lim_{y\to 0^{+}}U(x,y)=U(x,0)= Tr(U)\quad\mbox{in }H^s(\mathbb{R}^n).\end{equation}

In the sequel, we look for the homogenized problem or the limit equation satisfied by $U\in H^1(\mathbb{R}^{n+1}_{+},y^{1-2s})$.
To this end, we first observe that the flux quantity $\sigma_\eps(x,y)=y^{1-2s}\widetilde{A}_\eps(x)\nabla_{x,y}U_\eps(x,y)$ is  uniformly bounded in $L^2(\mathbb{R}^{n+1}_{+},y^{1-2s})^{n+1}$ because $A_\eps(x)$ and $U_\eps\in H^1(\mathbb{R}^{n+1}_{+},y^{1-2s})$ are uniformly bounded in their respective spaces. Thus, upto a subsequence denoted by same  $\{\sigma_\eps\}_{\eps>0}$, the flux sequence has a weak limit in $L^2(\mathbb{R}^{n+1}_{+},y^{1-2s})^{n+1}$, called it $\sigma(x,y)$, that is,
\begin{equation}\label{t64}
\sigma_\eps(x,y)\rightharpoonup \sigma(x,y) \quad \mbox{ weakly in }L^2(\mathbb{R}^{n+1}_{+},y^{1-2s})^{n+1}.\end{equation}
Since $\nabla_{x,y}\cdot \sigma_\eps(x,y) =0$ in $\mathbb{R}^{n+1}_{+}$ for all $\eps>0$, and since due to  \eqref{t64} we have\[\nabla_{x,y}\cdot\sigma_\eps(x,y) \to \nabla_{x,y}\cdot\sigma(x,y)\quad \mbox{ strongly in $H^{-1}(\mathbb{R}^{n+1}_{+},y^{1-2s})$ },\] 
 we find that
\[ \nabla_{x,y}\cdot \sigma(x,y) =0\mbox{ in }\mathbb{R}^{n+1}_{+}.\]

Hence, our ongoing job is reduced to find the relation between $\sigma\in L^2(\mathbb{R}^{n+1}_{+})^{n+1}$ and $U\in H^1(\mathbb{R}^{n+1}_{+},y^{1-2s})$, as usual is done in the homogenization framework.


Let us now use the framework of $H$-convergence (for
more details, see  \cite{MT-Hconv,T}) and  prove the following result: 
\begin{lem}
Let us consider the sequence $\{A_\eps\}_\eps$ satisfying conditions \eqref{t51} and that $H$-converges to $A_{*}$ (we denote by $A_\eps\xrightarrow{H} A_{*}$) that is,  
\[A_\eps\nabla w_\eps \rightharpoonup A_{*}\nabla w\quad \mbox{ weakly in } L^2(\mathbb{R}^n)^{n},\] for all test sequences $w_\eps\in H^1(\mathbb{R}^n)$ satisfying 
\begin{align*}
w_{\eps} &\rightharpoonup w \quad\mbox{weakly in }H^1(\mathbb{R}^n),\\
-\nabla\cdot(A_\eps\nabla w_\eps)& \quad\mbox{ strongly convergent in } H^{-1}(\mathbb{R}^n).
\end{align*}
 Then, we have 
 \begin{equation}\label{t71}
 y^{1-2s}\widetilde{A}_\eps(x)\nabla_{x,y}U_\eps(x,y)\rightharpoonup y^{1-2s}\widetilde{A}_{*}(x)\nabla_{x,y}U(x,y) \mbox{ weakly in }L^2(\mathbb{R}^{n+1}_{+},y^{1-2s})^{n+1},
\end{equation}
where  
$U\in H^1(\mathbb{R}^{n+1}_{+},y^{1-2s})$ solves the following homogenized problem: 
\begin{equation}\label{t74}
\begin{cases}
\mathscr{L}_{\widetilde{A}_{*}}^{1-2s}U=0 & \mbox{ in }\mathbb{R}_{+}^{n+1},
\\
U(\cdot,0)=u(\cdot) & \mbox{ in }\mathbb{R}^{n},
\end{cases}
\end{equation}
where $u\in H^s(\mathbb{R}^n)$ and 
\[
\widetilde{A}_{*}(x)=\left(\begin{array}{cc}
A_{*}(x) & 0\\
0 & 1
\end{array}\right).\]
\end{lem}
\begin{proof}
Let us consider the  region
$D_{(\delta,\delta^{-1})}=\{(x,y): x\in\mathbb{R}^{n}\mbox{ and }\delta<y<\delta^{-1}\}$,
for any $\delta>0$. Since the weight $y^{1-2s}$ is smooth enough  and
positive in $\overline{D_{(\delta,\delta^{-1})}}$, then $U_\eps \in H^{1}(D_{(\delta,\delta^{-1})})$ can be seen
as the solution of the following uniformly elliptic equation: 
\begin{equation}
\nabla_{x,y}\cdot\big(y^{1-2s}\widetilde{A}_\eps(x)\nabla_{x,y}U_\eps\big)=0\quad\mbox{ in } D_{(\delta,\delta^{-1})}.\label{t57}
\end{equation}We also get that $\|U_\eps\|_{H^1(D_{(\delta,\delta^{-1})})}$ is uniformly bounded w.r.t. $\eps$, and using  \eqref{t83} and \eqref{t97}, it follows
that\[\|U_\eps\|_{H^1(D_{(\delta,\delta^{-1})})}\leq C_\delta\big(\|f\|_{\widetilde{H}^{s}(\mathcal{O})^{*}}+\|g\|_{H^{s}(\mathbb{R}^{n})}\big).\] 

Thus, upto a subsequence denoted by same $U_\eps$, the sequence weakly converges to some limit $V$ in $H^1(D_{(\delta,\delta^{-1})})$. We claim that 
$$V=U|_{D_{(\delta,\delta^{-1})}},$$ where $U\in H^1(\mathbb{R}^{n+1}_{+},y^{1-2s})$ is the weak limit of $\{U_\eps\}_{\eps>0}$ in $H^1(\mathbb{R}^{n+1}_{+},y^{1-2s})$ introduced in \eqref{t91}. In fact, this claim simply follows from \eqref{t91} because 
\[ \int_{D_{(\delta,\delta^{-1})}} \varphi\, U_\eps \to \int_{D_{(\delta,\delta^{-1})}}\varphi\, U\quad\forall \varphi\in C^\infty_c(D_{(\delta,\delta^{-1})}),\mbox{ as }\eps\to 0. \]

Let us now claim that if $A_\eps \xrightarrow{H} A_{*}$, then $B_\eps(x,y)= y^{1-2s}\widetilde{A}_\eps(x)$ has the following $H$-limit:
\begin{equation}\label{t73} B_\eps(x,y)\xrightarrow{H} B_{*}(x,y)=y^{1-2s}\widetilde{A}_{*}(x) \mbox{ in }D_{(\delta,\delta^{-1})}.
\end{equation}
In fact, since $U_\eps\in H^1(D_{(\delta,\delta^{-1})})$ solves 
\begin{multline}
\label{t72}
(\mathcal{L_\eps})_{x}\,U_\eps(x,y)
=\frac{1-2s}{y}({U_\eps})_{y}(x,y)+({U_\eps})_{yy}(x,y)
\\
=y^{-1+2s}\partial_{y}\big(y^{1-2s}\partial_{y}U_\eps(x,y)\big)   \mbox{ in }D_{(\delta,\delta^{-1})}
\end{multline}
and 
since $U_\eps \rightharpoonup U $ weakly in $H^1(D_{(\delta,\delta^{-1})})$, then we claim the right hand side of \eqref{t72} satisfies 
\begin{equation}\label{tn1}
y^{-1+2s}\partial_{y}\big(y^{1-2s}\partial_{y}U_\eps(x,y)\big) \to y^{-1+2s}\partial_{y}\big(y^{1-2s}\partial_{y}U(x,y)\big) \mbox{ strongly in }H^{-1}(D_{(\delta,\delta^{-1})}).
\end{equation}
Proof of the claim \eqref{tn1}:
Note that, as the strip $D_{(\delta,\delta^{-1})}$ is bounded in $y$-direction, so by applying the standard Rellich compactness theorem (see \cite{grubb}) from  $U_\eps \rightharpoonup U $ weakly in $H^1(D_{(\delta,\delta^{-1})})$, we get $\partial_{y}U_\eps(x,y)\to \partial_{y}U(x,y)$ strongly in $L^2(D_{(\delta,\delta^{-1}))}$. So, 
 $y^{1-2s}\partial_{y}U_\eps(x,y)\to y^{1-2s}\partial_{y}U(x,y)$ strongly in $L^2(D_{(\delta,\delta^{-1}))}$. Following that, we have $$y^{-1+2s}\partial_{y}\big(y^{1-2s}\partial_{y}U_\eps(x,y)\big) \rightharpoonup y^{-1+2s}\partial_{y}\big(y^{1-2s}\partial_{y}U_\eps(x,y)\big)$$ weakly in $L^2(D_{(\delta,\delta^{-1})})$, 
therefore in $H^{-1}$ strong topology; that is, for any $\phi(x,y)\in C^\infty_c(D_{(\delta,\delta^{-1})})$, we have
\begin{multline*}
\int_{D_{(\delta,\delta^{-1})}} y^{-1+2s}\partial_{y}\big(y^{1-2s}\partial_{y}U_\eps(x,y)\big) \phi(x,y) \, dx\, dy
\\  
=\int_{D_{(\delta,\delta^{-1})}} \big(y^{1-2s}\partial_{y}U_\eps(x,y)\big)\, \partial_{y} \big( y^{-1+2s}\phi(x,y)\big) \, dx\, dy 
\\
\to  
\int_{D_{(\delta,\delta^{-1})}} \big(y^{1-2s}\partial_{y}U(x,y)\big) \partial_{y}\big( y^{-1+2s}\phi(x,y)\big) \, dx\, dy 
\\= \int_{D_{(\delta,\delta^{-1})}} y^{-1+2s}\partial_{y}\big(y^{1-2s}\partial_{y}U(x,y)\big) \phi(x,y) \, dx\, dy. 
\end{multline*}
This establishes our above claim \eqref{tn1}.

Thus, by passing to the limit in \eqref{t72}, as $\eps\to 0$, we obtain the following homogenized equation: 
\[
(\mathcal{L_{*}})_{x}\,U(x,y)=\frac{1-2s}{y}{U}_{y}(x,y)+{U}_{yy}(x,y) \quad \mbox{ in }D_{(\delta,\delta^{-1})},
\]
where $({\mathcal{L}_{*}})_{x}= -\nabla_x\cdot\left(A_{*}(x)\nabla_x\right)$. Moreover, we get the flux convergence 
  \[y^{1-2s}\widetilde{A}_\eps(x)\nabla_{x,y}U_\eps(x,y)\rightharpoonup y^{1-2s}\widetilde{A}_{*}(x)\nabla_{x,y}U(x,y) \mbox{ weakly in }L^2(D_{(\delta,\delta^{-1})})^{n+1}.\]
Thus, $U\in H^1(D_{(\delta,\delta^{-1})})$ solves 
\[
\nabla_{x,y}\cdot\big(y^{1-2s}\widetilde{A}_{*}(x)\nabla_{x,y}U(x,y)\big)=0\mbox{ in } D_{(\delta,\delta^{-1})},\]
which concludes \eqref{t73}.

Since \eqref{t73} holds for any $\delta>0$ small enough, and $y^{1-2s}\widetilde{A}_{*}(x)\nabla_{x,y}U(x,y)\in L^2(\mathbb{R}_{+}^{n+1},y^{1-2s})^{n+1}$, we would like to claim that $U\in H^1(\mathbb{R}^{n+1}_{+},y^{1-2s})$ is solution of 
\begin{equation}\label{tn2}
\nabla_{x,y}\cdot\big(y^{1-2s}\widetilde{A}_{*}(x)\nabla_{x,y}U\big)=0\quad\mbox{ in } \mathbb{R}^{n+1}_{+}=\underset{\delta>0}{\cup}\, D_{(\delta,\delta^{-1})}
\end{equation}
and we have the flux convergence \eqref{t71} in $L^2(\mathbb{R}_{+}^{n+1},y^{1-2s})^{n+1}$.
In order to justify our claim \eqref{tn2},  we need to show $U\in H^1(\mathbb{R}^{n+1}_{+},y^{1-2s})$ satisfy  
\[
\int_{\mathbb{R}^{n+1}_{+}} \big(y^{1-2s}\widetilde{A}_{*}(x)\nabla_{x,y}U(x,y)\big)\cdot\nabla_{x,y}\varphi(x,y)\, dx\,dy=0 \quad\mbox{for all }\varphi \in C^\infty_c(\mathbb{R}^{n+1}_{+}).
\]
Since $\varphi \in C^\infty_c(\mathbb{R}^{n+1}_{+})$ implies there exists  $\delta>0$ such that $\varphi \in C^\infty_c(D_{(\delta,\delta^{-1})})$, and then, from \eqref{t73}, we have 
\begin{multline*}
\int_{\mathbb{R}^{n+1}_{+}} \big(y^{1-2s}\widetilde{A}_{*}(x)\nabla_{x,y}U(x,y)\big)\cdot\nabla_{x,y}\varphi(x,y)\, dx\,dy
\\
=\int_{D_{(\delta,\delta^{-1})}} \big(y^{1-2s}\widetilde{A}_{*}(x)\nabla_{x,y}U(x,y)\big)\cdot\nabla_{x,y}\varphi(x,y)\, dx\,dy = 0.
\end{multline*}

Since $\sigma(x,y)$ is the weak limit of $\sigma_\eps(x,y)$ in $L^2(\mathbb{R}_{+}^{n+1},y^{1-2s})^{n+1}$ (see \eqref{t64}), and $$y^{1-2s}\widetilde{A}_{*}(x)\nabla_{x,y}U(x,y)\in L^2(\mathbb{R}_{+}^{n+1},y^{1-2s})^{n+1},$$ then  
\begin{equation}\label{tn3}
\sigma(x,y)=y^{1-2s}\widetilde{A}_{*}(x)\nabla_{x,y}U(x,y) \mbox{ in } \mathbb{R}_{+}^{n+1},
\end{equation}
that is, we want to show that
\begin{multline*} 
\int_{\mathbb{R}^{n+1}_{+}} \big(y^{1-2s}\widetilde{A}_{\eps}(x)\nabla_{x,y}U_\eps(x,y)\big)\, \varphi(x,y)\, dx\,dy 
\\
\to 
\int_{\mathbb{R}^{n+1}_{+}} \big(y^{1-2s}\widetilde{A}_{*}(x)\nabla_{x,y}U(x,y)\big)\, \varphi \, dx\,dy \quad\mbox{for all }\varphi\in C^\infty_c(\mathbb{R}^{n+1}_{+}).
\end{multline*}
Let us show our claim \eqref{tn3} in a similar way. Since $\varphi \in C^\infty_c(\mathbb{R}^{n+1}_{+})$ implies there exists $\delta>0$ such that $\varphi \in C^\infty_c(D_{(\delta,\delta^{-1})})$, so from 
$\sigma(x,y)$ is the weak limit of $\sigma_\eps(x,y)$ in $L^2(D_{(\delta,\delta^{-1})})$, we have the desired conclusion \eqref{tn3}.

Then combining with the fact $U(x,0)=u(x)\in H^s(\mathbb{R}^n)$ due to \eqref{t89}, we establish that the homogenized boundary value problem \eqref{t74} is satisfied by $U\in H^1(\mathbb{R}^{n+1}_{+},y^{1-2s})$. 

\hfill\end{proof}

\section{Proof of Theorem \ref{t9}}
\label{sect4}
We first use Proposition \ref{t16} and we get 
\begin{equation}
\lim_{y\to 0^{+}}\dfrac{U_\eps(\cdot,y)-U_\eps(\cdot,0)}{y^{2s}}=\frac{1}{2s}\lim_{y\to0^{+}}y^{1-2s}\partial_{y}U_\eps(\cdot,y)=\frac{\Gamma(-s)}{4^{s}\Gamma(s)}\mathcal{L}^{s}_\eps u_\eps \quad\mbox{in $H^{-s}(\mathbb{R}^{n}),$}\label{t52}
\end{equation}
where $U_\eps\in H^{1}(\mathbb{R}_{+}^{n+1},y^{1-2s})$  is the weak
solution of problem \eqref{t38}, and we also have
\begin{equation}
\lim_{y\to0^{+}}\dfrac{U(\cdot,y)-U(\cdot,0)}{y^{2s}}=\frac{1}{2s}\lim_{y\to0^{+}}y^{1-2s}\partial_{y}U(\cdot,y)=\frac{\Gamma(-s)}{4^{s}\Gamma(s)}\mathcal{L}^{s}_{*} u \quad\mbox{in $H^{-s}(\mathbb{R}^{n}),$}\label{t54}
\end{equation}
where $U\in H^{1}(\mathbb{R}_{+}^{n+1},y^{1-2s})$  is the weak
solution of the homogenized problem \eqref{t74}.

Next, due to \eqref{t42} and \eqref{t15}, we have 
\begin{multline*}
\Big(\lim_{y\to0^{+}}y^{1-2s}\partial_{y}U_\eps(x,y),\phi(x,0)\Big)_{H^{-s}(\mathbb{R}^{n})\times H^{s}(\mathbb{R}^{n})}
\\
=\int_{\mathbb{R}_{+}^{n+1}}y^{1-2s}\widetilde{A}_\eps(x)\nabla_{x,y}U_\eps\cdot\nabla_{x,y}\phi\, dxdy
\end{multline*}
and 
\begin{multline*}
\Big(\lim_{y\to 0^{+}}y^{1-2s}\partial_{y}U(x,y),\phi(x,0)\Big)_{H^{-s}(\mathbb{R}^{n})\times H^{s}(\mathbb{R}^{n})}
\\
=\int_{\mathbb{R}_{+}^{n+1}}y^{1-2s}\widetilde{A}(x)\nabla_{x,y}U\cdot\nabla_{x,y}\phi\, dxdy,
\end{multline*}
for all $\phi\in H^{1}(\mathbb{R}_{+}^{n+1},y^{1-2s})$. 

Let us now pass to the limit in the above identity, as $\eps\to 0$,  and use the flux convergence \eqref{t71} to find that
\begin{multline}\label{t36}
\Big(\lim_{y\to0^{+}}y^{1-2s}\partial_{y}U_\eps(x,y),\phi(x,0)\Big)_{H^{-s}(\mathbb{R}^{n})\times H^{s}(\mathbb{R}^{n})}
\\
\to \Big(\lim_{y\to0^{+}}y^{1-2s}\partial_{y}U(x,y),\phi(x,0)\Big)_{H^{-s}(\mathbb{R}^{n})\times H^{s}(\mathbb{R}^{n})},
\end{multline}
for all $\phi\in H^{1}(\mathbb{R}_{+}^{n+1},y^{1-2s})$. 

Then, by taking $\phi(x,0)=\psi(x)\in C^\infty_c(\mathbb{R}^n)$ (which is clearly possible), and  using  
\eqref{t52}, \eqref{t54} and \eqref{t36}, we obtain  
\begin{equation}\label{t80} 
\big(\mathcal{L}^{s}_{\eps}u_\eps(x),\psi(x)\big)_{H^{-s}(\mathbb{R}^{n})\times H^{s}(\mathbb{R}^{n})}
\\
\to \big(\mathcal{L}^{s}_{*}u(x),\psi(x)\big)_{H^{-s}(\mathbb{R}^{n})\times H^{s}(\mathbb{R}^{n})} \mbox{ as }\eps \to 0,
\end{equation}
for all $\psi\in C^\infty_c(\mathbb{R}^n)$. 

We now choose $Supp\, \psi \subset \mathcal{O}$, and since  $\mathcal{L}^{s}_{\eps}u_\eps=f$ in $\mathcal{O}$ due to \eqref{t79},  we therefore obtain
\[ \mathcal{L}^s_{*}u=f \quad\mbox{ in }\mathcal{O}.\]
On the other hand, since $u_\eps =g$ in $\mathcal{O}_e$, then as a $H^s(\mathbb{R}^n)$-weak limit of the sequence $\{u_\eps\}_{\eps>0}$, we get 
\[u=g \quad \mbox{ in }\mathcal{O}_e.\] 
We thus obtained the homogenized equation: $u\in H^s(\mathbb{R}^n)$ is the unique solution of the following non-local problem:
 \[
\begin{cases}
\mathcal{L}_{*}^{s}u= (-\nabla\cdot(A_{*}(x)\nabla )^s u=f & \mbox{ in \ensuremath{\mathcal{O}}},\\
u=g & \mbox{ in }\mathcal{O}_{e},
\end{cases}
\]
for $f\in \widetilde{H}^{s}(\mathcal{O})^{*}$ and $g\in H^s(\mathbb{R}^n)$.

Let us now prove the energy convergence, precisely 
\begin{equation}\label{t3} 
\|\mathcal{L}^{s/2}_\eps u_\eps\|_{L^2(\mathbb{R}^n)} \to \|\mathcal{L}^{s/2}_{*}u \|_{L^2(\mathbb{R}^n)}, \quad \mbox{as } \eps\to 0.
\end{equation}
To this end, we multiply \eqref{t38} by $U_\eps\in H^1(\mathbb{R}^{n+1}_{+},y^{1-2s})$, we integrate by parts, and we pass to the limit as $\eps\to 0$, obtaining
\begin{equation} 
\int_{\mathbb{R}_{+}^{n+1}}y^{1-2s}\widetilde{A}_\eps(x)\nabla_{x,y}U_\eps\cdot\nabla_{x,y}U_\eps\, dxdy \to \int_{\mathbb{R}_{+}^{n+1}}y^{1-2s}\widetilde{A}_{*}(x)\nabla_{x,y}U\cdot\nabla_{x,y}U\, dxdy.
\end{equation}
Consequently, the convergence \eqref{t3} holds.

Finally, we are left out to show the flux convergence \eqref{t21}. To this end, let us observe that our analysis simply suggests that, if we take $\frac{s}{2}$ instead of $s\in (0,1)$ in \eqref{t38}, which is clearly possible and independent of the main problem \eqref{t79} to consider, then it follows from 
\eqref{t80} that\[ 
\mathcal{L}^{s/2}_\eps u_\eps \rightharpoonup \mathcal{L}^{s/2}_{*}u \quad\mbox{ weakly in }H^{-s/2}(\mathbb{R}^n).
\]
Since
$\{\mathcal{L}^{s/2}_\eps u_\eps \}_{\eps>0}\subseteq L^2(\mathbb{R}^n)$  has a $L^2(\mathbb{R}^n)$ weak sub-sequential limit $v\in L^2(\mathbb{R}^n)$ (see \eqref{t32}) and $\mathcal{L}^{s/2}_{*}u\in L^2(\mathbb{R}^n)$, thus $v=\mathcal{L}^{s/2}_{*}u.$ Therefore, 
\[\mathcal{L}^{s/2}_\eps u_\eps \rightharpoonup \mathcal{L}^{s/2}_{*}u \quad \mbox{ weakly in }L^2(\mathbb{R}^n).\]
This completes the proof of  Theorem \ref{t9}.

\section{Non-local homogenization in perforated domain}
\label{sect5}

Let us consider the sequence of closed subsets $\{T_\eps\}_{\eps>0}$  which are called holes and the  perforated domain $\mathcal{O}_\eps$ defined in \eqref{t98} with the condition on the Lebesgue measure \eqref{t99}. 

For $s\in(0,1)$ and for  $\eps>0$, we consider the following non-local Dirichlet problem associated with the fractional Laplace operator described in \eqref{t10001}, precisely:
\begin{equation*}
\begin{cases}
(-\Delta)^s u_\eps=f & \mbox{ in \ensuremath{\mathcal{O}_\eps}},
\\
u_\eps=g & \mbox{ in }\mathbb{R}^n\setminus\mathcal{O}_\eps,
\end{cases}
\end{equation*}
for some $f\in\widetilde{H}^{s}(\mathcal{O})^{*}$ and $g\in H^{s}(\mathbb{R}^{n})$.
We recall from \eqref{t66} that, for $v\in H^s(\mathbb{R}^n)$,
\[ (-\Delta)^sv(x) = c_{n,s}\, \mbox{P.V.}\int_{\mathbb{R}^n}\frac{v(x)-v(y)}{|x-y|^{n+2s}}\, dy,\]
with $c_{n,s}= \frac{\Gamma(\frac{n}{2}+s)}{|\Gamma(-s)|}\,\frac{4^s}{\pi^{n/2}}$,
and $\mbox{P.V.}$ stands for the standard principal value operator. Then, we define the bilinear form as: for any $v,\mbox{ }w\in H^{s}(\mathbb{R}^{n})$,
\begin{equation}\begin{aligned}
\mathcal{B}^s(v,w):&=c_{n,s}\int_{\mathbb{R}^{n}}\int_{\mathbb{R}^{n}}\frac{(v(x)-v(z))(w(x)-w(z))}{|x-z|^{n+2s}}\,dx\,dz\\
&=\int_{\mathbb{R}^{n}}(-\Delta)^{s/2}v\,(-\Delta)^{s/2}w\, dx.\label{t68}
\end{aligned}\end{equation}
Then for each fixed $\eps>0$, there exists a unique solution $u_\eps\in H^s(\mathbb{R}^n)$ such that\begin{equation}
\mathcal{B}^s(u_\eps,w)=\left\langle f,w\right\rangle \mbox{ for any }w\in\widetilde{H}^{s}(\mathcal{O}_\eps)\mbox{ with }u_\eps-g\in\widetilde{H}^{s}(\mathcal{O}_\eps),\label{t5}
\end{equation}
for any $f\in\widetilde{H}^{s}(\mathcal{O})^{*}$ and $g\in H^{s}(\mathbb{R}^{n})$,
where $\left\langle \cdot,\cdot\right\rangle $ stands for the duality
pairing between $(\widetilde{H}^{s})^{*}$ and $\widetilde{H}^{s}$. The above existence and uniqueness result is a direct consequence of  Proposition \ref{t94}.

Let us first note that $u_\eps\in H^s(\mathbb{R}^n)$ is already defined everywhere in the entire space. We now consider 
the bilinear form \eqref{t5} with $w=u_\eps-g\in H^s(\mathbb{R}^n)$ and use definition \eqref{t68} in order to get
\[
\int_{\mathbb{R}^{n}}(-\Delta)^{s/2}u_\eps\,(-\Delta)^{s/2}u_\eps\, dx -\int_{\mathbb{R}^{n}}(-\Delta)^{s/2}u_\eps\,(-\Delta)^{s/2}g\, dx  
=\left\langle f,u_\eps-g\right\rangle, 
\]
or
\begin{equation}\label{t67}
\frac{1}{2}\big\|(-\Delta)^{s/2}u_\eps\big\|^2_{L^2(\mathbb{R}^n)} \leq \ \frac{1}{2}\big\|(-\Delta)^{s/2}g\big\|^2_{L^2(\mathbb{R}^n)} +\|f\|_{\widetilde{H}^{s}(\mathcal{O})^{*}}\|u_\eps-g\|_{H^s(\mathbb{R}^n)}.
\end{equation}
Since $\mathcal{O}$ is bounded, $u_\eps-g=0$ in $\mathcal{O}_e$, and the Hardy-Littlewood-Sobolev inequality (see \cite{Stein}) holds, we obtain  
\[
\|u_\eps-g\|_{L^2(\mathbb{R}^n)} \leq C_{\mathcal{O}} \|u_\eps-g\|_{L^{\frac{2n}{n-2s}}(\mathbb{R}^n)} \leq C \big\|(-\Delta)^{s/2}\, (u_\eps-g)\big\|_{L^2(\mathbb{R}^n)},
\]
then, due to estimate \eqref{t67} we deduce that  
\[\|u_\eps\|_{H^s(\mathbb{R}^n)}\leq C\big(\|f\|_{\widetilde{H}^{s}(\mathcal{O})^{*}}+\|g\|_{H^s(\mathbb{R}^n)}\big),
\]
where the constant $C$ is independent of $\eps>0$. Consequently, upto a subsequence, still denoted by $\{u_\eps\}_{\eps>0}$, we have
\begin{equation}\label{t53}
u_\eps \rightharpoonup u \quad \mbox{ weakly in } H^s(\mathbb{R}^n),
\end{equation}
for some $u\in H^s(\mathbb{R}^n)$.
 
Our goal is to find the  problem satisfied by the limit $u\in H^s(\mathbb{R}^n)$, precisely the homogenized problem.
To accomplish this we 
recall the standard homogenization framework for the Laplacian operator in perforated domain,  following the work of Cioranescu and Murat in \cite[Chapter 4]{MT}.
\subsection*{Homogenization framework in perforated domain}
  
Let us assume that there exist a sequence of functions $\{w_\eps\}_{\eps>0}$ such that satisfies the following three hypotheses: 
\begin{enumerate}
\item[(H1)] $w_\eps\in H^1(\mathcal{O})$;
\item[(H2)] $w_\eps=0$ on the holes $\underset{0<\delta\leq\eps}{\cup}T_\delta$;
\item[(H3)] $w_\eps \rightharpoonup 1$ weakly in $H^1(\mathcal{O})$.
\end{enumerate}
Such sequences exist and can be constructed for the inhomogeneities governed by spherical, elliptical, cylindrical holes, etc., in dimension $n\geq 2$. Let us recall here one example of such sequences (for more details, see \cite[Chapter 4]{MT}):
\begin{example}[Spherical holes periodically distributed in
volume]\label{tn6}
For each value of $\eps>0$, one covers $\mathbb{R}^n$ $(n\geq 2)$ by cubes $Y_\eps$ of size $2\eps$. From
each cube we remove the ball $T_\eps$ of radius $a_\eps>0$ and both, cube and ball, share the same center. In this way, $\mathbb{R}^n$ is perforated by spherical identical
holes as
\[ \mathcal{O}_\eps =\mathcal{O}\cap \Big(\mathbb{R}^n\setminus \underset{0<\delta\leq\eps}{\cup}\, T_\delta\Big),\]
which means that we remove from $\mathcal{O}$ small balls of radius $a_\eps$, whose centers
are the nodes of a lattice in $\mathbb{R}^n$ with cell size $2\eps$. 
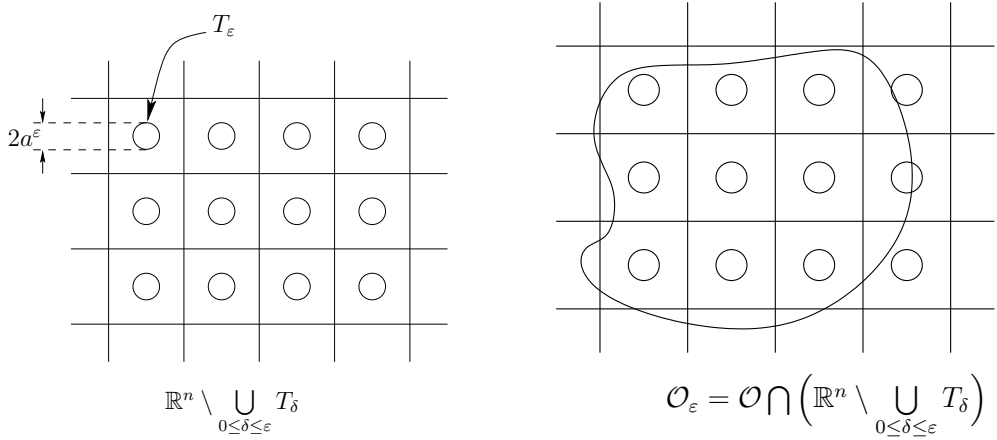
\begin{figure}[ht!!]    
\centerline{
   \resizebox{6cm}{!}
   {
   \begin{picture}(0,0)%
\includegraphics{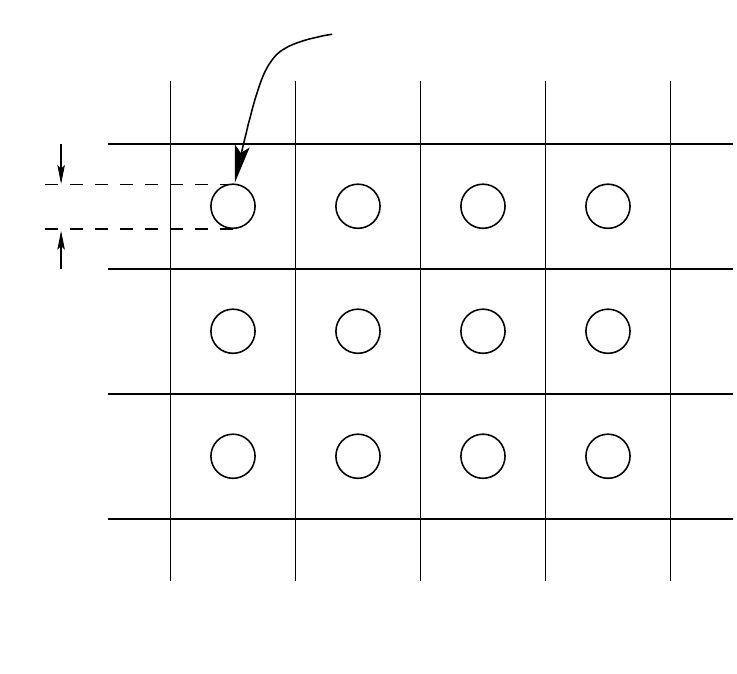}%
\end{picture}%
\setlength{\unitlength}{3947sp}%
\begingroup\makeatletter\ifx\SetFigFont\undefined%
\gdef\SetFigFont#1#2#3#4#5{%
  \reset@font\fontsize{#1}{#2pt}%
  \fontfamily{#3}\fontseries{#4}\fontshape{#5}%
  \selectfont}%
\fi\endgroup%
\begin{picture}(3530,3218)(3983,-6205)
\put(5629,-3134){\makebox(0,0)[lb]{\smash{{\SetFigFont{12}{14.4}{\rmdefault}{\mddefault}{\updefault}{\color[rgb]{0,0,0}$T_\varepsilon$}%
}}}}
\put(5251,-6136){\makebox(0,0)[lb]{\smash{{\SetFigFont{12}{14.4}{\rmdefault}{\mddefault}{\updefault}{\color[rgb]{0,0,0}$\R^n\setminus\bigcup\limits_{0\leq\delta\leq\varepsilon}T_\delta$}%
}}}}
\put(3998,-4039){\makebox(0,0)[lb]{\smash{{\SetFigFont{12}{14.4}{\rmdefault}{\mddefault}{\updefault}{\color[rgb]{0,0,0}$2a^\varepsilon$}%
}}}}
\end{picture}%
}
\hspace{1cm}
    \resizebox{6cm}{!}
    {
    \begin{picture}(0,0)%
\includegraphics{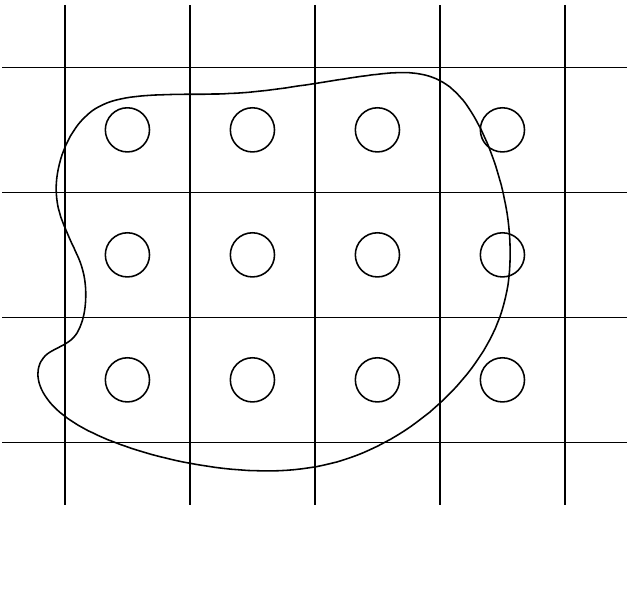}%
\end{picture}%
\setlength{\unitlength}{3947sp}%
\begingroup\makeatletter\ifx\SetFigFont\undefined%
\gdef\SetFigFont#1#2#3#4#5{%
  \reset@font\fontsize{#1}{#2pt}%
  \fontfamily{#3}\fontseries{#4}\fontshape{#5}%
  \selectfont}%
\fi\endgroup%
\begin{picture}(3024,2856)(4489,-6205)
\put(5251,-6136){\makebox(0,0)[lb]{\smash{{\SetFigFont{12}{14.4}{\rmdefault}{\mddefault}{\updefault}{\color[rgb]{0,0,0}$
\mathcal{O}_\varepsilon=\mathcal{O}\bigcap\Big(\R^n\setminus\bigcup\limits_{0\leq\delta\leq\varepsilon}T_\delta\Big)$}%
}}}}
\end{picture}%
    }
}
\caption{Spherical holes periodically distributed in
volume.}
\label{Fig1} 
\end{figure}


In this case, one constructs $w_\eps$ in polar coordinates in the annulus $B_\eps\setminus T_\eps$ as
follows: 
\begin{equation*}
w_\eps(r) = 
\left\{
\begin{array}{lll} 
\displaystyle\frac{\ln a_\eps -\ln r}{\ln a_\eps-\ln \eps}&\mbox{ if }n=2,
\\[2ex]
\displaystyle\frac{a_\eps^{-(n-2)} -r^{-(n-2)}}{a_\eps^{-(n-2)}-\eps^{-(n-2)}}&\mbox{ if }n\geq 3,
\end{array}
\right.
\end{equation*}
 where $r=|x|$. 
\end{example}

\subsection*{Proof of  Theorem \ref{t63}}
By hypotheses (H1) and (H2), for any $\varphi\in\mathcal{D}(\mathcal{O})$, the sequence $\{w_\eps\varphi\}_{\eps>0}\subseteq \widetilde{H}^s(\mathcal{O})$ with $Supp \, w_\eps\varphi \subset\overline{\mathcal{O}}$ (or lies in $H^s_{\overline{\mathcal{O}}}(\mathbb{R}^n)$). Thus,
one can take $w_\eps\varphi\in \widetilde{H}^s(\mathcal{O})$ as test function  in the variational formulation \eqref{t5} to obtain  
\begin{equation}\label{t46} \int_{\mathbb{R}^{n}}(-\Delta)^{s/2}u_\eps\,(-\Delta)^{s/2}(w_\eps\varphi)\, dx  =\langle f, w_\eps\varphi\rangle_{\widetilde{H}^{s}(\mathcal{O})^{*},\,\widetilde{H}^s(\mathcal{O})}.
\end{equation}
Since $u_\eps$ weakly converges to $u$ in $H^s(\mathbb{R}^n)$ due to \eqref{t53}, then
\begin{equation}\label{t44}
(-\Delta)^{s/2}u_\eps \rightharpoonup (-\Delta)^{s/2}u \quad \mbox{ weakly in }L^2(\mathbb{R}^n).
\end{equation}
Next,  due to hypothesis (H3), we have
\begin{equation}\label{t7} w_\eps\varphi \rightharpoonup \varphi \quad \mbox{ weakly in }H^1_{\overline{\mathcal{O}}}(\mathbb{R}^n).\end{equation}
At this point, we recall the Rellich Theorem from \cite[Theorem 8.2, pp.199]{grubb}, which says that: for $t_1 < t_2$ with $t_1,t_2\in\mathbb{R}$, the inclusion map 
\[H^{t_2}_{\overline{\mathcal{O}}}(\mathbb{R}^n) \xhookrightarrow{}
H^{t_1}(\mathbb{R}^n)\mbox{  is compact},\]
then choosing $t_1=s\in(0,1)$  and $t_2=1$, \eqref{t7} gives us the following strong convergence:
\[w_\eps\varphi \to \varphi \quad \mbox{ strongly in }H^s(\mathbb{R}^n), \mbox{ as }\eps\to 0.\]
Thus, we deduce
\begin{equation}\label{t8}
(-\Delta)^{s/2}w_\eps\varphi \to (-\Delta)^{s/2}\varphi \quad \mbox{ strongly in }L^2(\mathbb{R}^n).
\end{equation}

We can now pass to the limit in identity \eqref{t46} and by using the weak convergence \eqref{t44} and the strong convergence \eqref{t8}, we obtain 
\[\int_{\mathbb{R}^{n}}(-\Delta)^{s/2}u\,(-\Delta)^{s/2}\varphi\, dx = \langle f, \varphi\rangle_{\widetilde{H}^{s}(\mathcal{O})^{*},\, \widetilde{H}^s(\mathcal{O})},\]
which can be written as follows: 
\[\mathcal{B}^s(u,\varphi)=\langle f,\varphi\rangle \mbox{ for any }\varphi\in\widetilde{H}^{s}(\mathcal{O})\mbox{ with }u-g\in\widetilde{H}^{s}(\mathcal{O}).\]

Hence, $u\in H^s(\mathbb{R}^n)$ uniquely solves the homogenized equation
\[\begin{cases}
(-\Delta)^su  = f &\quad\mbox{in }\mathcal{O},
\\
u=g    &\quad\mbox{in }\mathbb{R}^n\setminus\mathcal{O},
\end{cases}
\]
for $f\in\widetilde{H}^{s}(\mathcal{O})^{*}$ and $g\in H^{s}(\mathbb{R}^{n})$.
This completes the proof of  Theorem \ref{t63}.
\hfill\qed
\begin{rem}
Let us observe that the  problem associated with  the local operator $-\Delta$, precisely 
\[ -\Delta u_\eps =f \quad\mbox{in }\mathcal{O}_\eps,\quad
 u_\eps\in H^1_0(\mathcal{O}_\eps),\]
with $f\in L^2(\mathcal{O}), $ the $H^1_0(\mathcal{O})$-weak limit,  as $\eps \to 0$,  of the extension sequence 
$\{\widetilde{u}_\eps =\chi_{\mathcal{O}_\eps}u_\eps\}_{\eps>0}\subseteq H^1_0(\mathcal{O})$, say $u\in H^1_0(\mathcal{O})$,
solves the following homogenized problem:
\[ -\Delta u +\mu u=f \quad\mbox{in }\mathcal{O},\quad
 u\in H^1_0(\mathcal{O}),\]
where the so-called "strange term" $\mu\in W^{-1,\infty}(\mathcal{O})$,  defined along the hypotheses (H1), (H2), (H3) and the following convergence:
for every sequence $v_\eps$ such that $v_\eps=0$ on $T_\eps$ satisfying $v_\eps\rightharpoonup v$ weakly in $H^1(\mathcal{O})$ (with $v\in H^1(\mathcal{O})$), one has 
\[\langle-\Delta w_\eps,\varphi v_\eps\rangle_{\left(H^{-1}(\mathcal{O}),H^1_0(\mathcal{O})\right)}\to \langle \mu, \varphi v\rangle_{\left(H^{-1}(\mathcal{O}),H^1_0(\mathcal{O})\right)},
\quad\mbox{ for all }\varphi\in\mathcal{D}(\mathcal{O}).\] Nevertheless, in our non-local problem we do not find  any additional term as $\mu\, u$. The key difference is that, under the same hypotheses on $\{w_\eps\}_{\eps>0}$, the strong convergence result \eqref{t8} fails whenever $s=1$. Since $(-\Delta)^{s}w\to (-\Delta)w$ in $L^2(\mathcal{O})$ for $w\in H^2(\mathcal{O})$ as $s\to 1^{-}$ (see \cite{EGV}), so this tells us the homogenization process as $\eps\to 0$ in perforated domain might not stable under $s\to 1^{-}$ unless $\mu\equiv 0$. In our previous Example \ref{tn6}, with $T_\eps$ as a periodic network of balls of radius $a_\eps$ and centered in $2\pi\eps\mathbb{Z}^N$, $\mu$ becomes $0$ only if (see \cite[Theorem 2.1]{DRL}):
$$\lim_{\eps\to 0} \frac{-(ln\, a_\eps)^{-1}}{\eps^2}=0, \mbox{ for $n=2$,}\qquad \lim_{\eps\to 0} \frac{a_\eps}{\eps^3}=0, \mbox{ for $n\geq 3$}.
$$
Then, in this case, we can say the limiting process as $\eps\to 0$ and $s\to 1^{-}$ are interchangeable.
\end{rem}



\def\cprime{$'$}

\end{document}